\documentclass[11pt]{amsart}

\usepackage[T1]{fontenc}
\usepackage{amsmath,amssymb,mathtools}
\usepackage{aliascnt}
\usepackage[hidelinks]{hyperref}
\usepackage[nameinlink,noabbrev]{cleveref}

\newtheorem{theorem}{Theorem}[section]
\newaliascnt{proposition}{theorem}
\newtheorem{proposition}[proposition]{Proposition}
\aliascntresetthe{proposition}
\newaliascnt{lemma}{theorem}
\newtheorem{lemma}[lemma]{Lemma}
\aliascntresetthe{lemma}
\newaliascnt{corollary}{theorem}
\newtheorem{corollary}[corollary]{Corollary}
\aliascntresetthe{corollary}
\theoremstyle{definition}
\newaliascnt{definition}{theorem}

\aliascntresetthe{definition}
\theoremstyle{remark}
\newaliascnt{remark}{theorem}
\newtheorem{remark}[remark]{Remark}
\aliascntresetthe{remark}
\numberwithin{equation}{section}

\crefname{theorem}{Theorem}{Theorems}
\crefname{proposition}{Proposition}{Propositions}
\crefname{lemma}{Lemma}{Lemmas}
\crefname{corollary}{Corollary}{Corollaries}
\Crefname{theorem}{Theorem}{Theorems}
\Crefname{proposition}{Proposition}{Propositions}
\Crefname{lemma}{Lemma}{Lemmas}
\Crefname{corollary}{Corollary}{Corollaries}

\newcommand{\R}{\mathbb R}
\newcommand{\E}{\textsf{E}}
\newcommand{\Pp}{\textsf{P}}
\newcommand{\Var}{\operatorname{Var}}
\newcommand{\cP}{\mathcal P}
\newcommand{\ip}[2]{\left\langle #1,#2\right\rangle}
\newcommand{\norm}[1]{\left\lVert #1\right\rVert}
\newcommand{\abs}[1]{\left\lvert #1\right\rvert}
\newcommand{\ot}{\otimes}
\newcommand{\1}{\mathbf 1}

\allowdisplaybreaks[2]

\hypersetup{
  pdftitle={Sharp Convex Concentration for Symmetric Random Tensors with Subgaussian Coordinates},
  pdfauthor={Xuanang Hu}
}

\title[Sharp Convex Concentration for Symmetric Random Tensors]{Sharp Convex Concentration for Symmetric Random Tensors with Subgaussian Coordinates}
\author{Xuanang Hu}
\address{Shandong University, Jinan, China.}
\email{xuananghu7@gmail.com}
\date{}

\begin{document}

\begin{abstract}
Let $X=(X_1,\ldots,X_n)$ have independent coordinates with
mean zero, variance one, and $\norm{X_i}_{\psi_2}\le K$, and let
$H_d=(\mathbb R^n)^{\otimes_2 d}$. Let $L>0$ and let $f:H_d\to\mathbb R$ be convex and
$L$-Lipschitz. We prove that, for $0\le t\le c_KLn^{d/2}$,
\[
 \Pp\left\{
 \abs{f(X^{\otimes d})-\E f(X^{\otimes d})}>t
 \right\}
 \le C\exp\left[-c_K\mathcal I_{n,d}\left(
 \frac{t}{L n^{(d-1)/2}}
 \right)\right],
\]
where
\[
 \mathcal I_{n,d}(s)=
 \min\left\{
 \frac{s^2}{d^2},
 \frac{s^2}{d\log(e+nd/s^2)}
 \right\},\qquad s>0,
 \qquad \mathcal I_{n,d}(0)=0.
\]
The first rate is forced by changes in $\norm X$. The second comes from
changes of $X$ when its norm is nearly fixed. The proof constructs one
coupling that controls both the coordinatewise conditional displacement and
the mean squared Euclidean distance, and combines these bounds with a
second-order estimate for $x\mapsto x^{\otimes d}$. The rate is minimax sharp, scale by scale, even when the subgaussian norms
are bounded by an absolute constant. For bounded coordinates the logarithm in the second rate disappears.
\end{abstract}

\maketitle

\section{Introduction}

Concentration of measure on product spaces is one of the basic sources of
high-dimensional probability estimates; general accounts may be found in
\cite{Ledoux,BLM}. Two classical ideas are especially relevant here.
Talagrand introduced a convex distance for product measures and proved a
Gaussian-type enlargement inequality
\cite[Section~4.1, Theorem~4.1.1]{Talagrand}. One consequence is Gaussian
concentration for convex Lipschitz functions of independent bounded
coordinates; a convenient formulation is \cite[Theorem~1.1]{Vershynin}.
A second line of work relates relative entropy to the distance needed to couple
two probability measures. Marton's information-theoretic proof of the
blowing-up lemma \cite{Marton1986} and her later $\bar d$-distance inequality
\cite{Marton1996} made this connection explicit. Talagrand's
transportation-cost inequality for Gaussian and product measures is another
classical result in this direction \cite{TalagrandTransport}. Later work
formalized costs that depend on the conditional distribution of one coupled
point given the other; see, for example, \cite{GozlanRobertoSamsonTetali}.
That viewpoint is close to the conditional displacement used below. Our proof
uses an explicit maximal coupling: relative entropy is converted directly into
bounds on the coordinatewise conditional displacement and on the mean squared
Euclidean distance.

Random tensors introduce a second difficulty: the tensor map amplifies motion
by a factor that depends on the degree. For a simple random tensor
\[
 x_1\ot\cdots\ot x_d,
\]
the factors are independent. Vershynin proved convex and Euclidean
concentration estimates with the correct dependence on the dimension and the
degree \cite[Theorems~1.3 and~1.4]{Vershynin}. In the bounded convex case,
the natural variance scale for a simple tensor is $d n^{d-1}$; see
\cite[Theorem~1.3]{Vershynin}. In the symmetric model the same vector appears
in every slot, and the example $f(T)=\norm T$ gives
$f(X^{\otimes d})=\norm X^d$. Changes of $\norm X$ therefore force the
larger variance scale $d^2n^{d-1}$. Thus the symmetric problem is not obtained
by simply identifying the independent factors.

Vershynin singled out the symmetric tensor
\[
 X^{\otimes d}=X\ot\cdots\ot X
\]
as an open case and noted that a direct decoupling argument is expected to lose
factors exponential in $d$ \cite[Section~1.5]{Vershynin}. More recent work
has obtained sharp bounds for sums of simple tensors by empirical-process and
chaining methods \cite[Theorem~2.1]{AlGhattasChenSanzAlonso}; Abdalla and
Vershynin later gave a shorter proof of the corresponding dimension-free
estimate \cite[Corollary~1.2]{AbdallaVershynin}. Related results for
asymmetric simple tensors are developed in \cite[Theorem~2.1]{ChenSanzAlonso}.
These papers concern independent tensor factors or sums of rank-one tensors.
The present paper concerns one symmetric tensor, where the same random vector
occurs in every slot. Throughout this paper, ``symmetric tensor'' refers
specifically to this repeated-vector model. Related concentration estimates for
non-isotropic random tensors, motivated by learning and empirical-risk
minimization, were obtained by Even and Massouli\'e \cite{EvenMassoulie}.

Chen and Sanz-Alonso study Frobenius-norm concentration for sample moment
tensors and Hilbert-valued polynomial moments
under Gaussian convex domination \cite{ChenSanzAlonsoFrobenius}. Those
results concern empirical moment tensors or polynomial structure. Here we
study arbitrary convex Lipschitz functionals of one repeated-vector tensor.

A related product-space result is due to Huang and Tikhomirov. For independent
subgaussian coordinates they proved the sharp dimension-dependent
concentration inequality for convex $1$-Lipschitz functions
\cite[Theorem~1.3]{HuangTikhomirov}, and they proved that its logarithmic factor
is necessary for the full subgaussian class
\cite[Proposition~1.4]{HuangTikhomirov}. The function
$x\mapsto f(x^{\otimes d})$ is generally not convex, so their theorem cannot be
applied directly to the symmetric tensor problem. Nevertheless, the same
product-space scale
\[
 u\log(e+n/u)
\]
reappears in the coupling estimate used below. The same distinction appears in our result: the logarithmic factor remains
for general subgaussian coordinates, while the bounded-coordinate theorem below has
no such factor.

The main geometric point is visible from the derivative of
\[
 \Phi_d(x)=x^{\otimes d}.
\]
For $h\in\R^n$, with the second term below omitted when $d=1$,
\begin{equation}\label{eq:intro-differential}
 \norm{D\Phi_d(x)h}^2
 =d\norm{x}^{2d-2}\norm{h}^2
 +d(d-1)\norm{x}^{2d-4}\ip{x}{h}^2.
\end{equation}
If $h$ is parallel to $x$, then
\[
 \norm{D\Phi_d(x)h}=d\norm{x}^{d-1}\norm h.
\]
If $h\perp x$, then
\[
 \norm{D\Phi_d(x)h}=\sqrt d\,\norm{x}^{d-1}\norm h.
\]
Thus one cannot use the larger factor $d$ for every change of $x$. A sharp
argument has to separate changes of the norm from changes that keep the norm
fixed, up to a quadratic error. At the typical radius $\norm X\asymp\sqrt n$,
this predicts the two scales
\begin{equation}\label{eq:intro-scale}
 n^{(d-1)/2}
 \left(d\sqrt u+\sqrt{d u\log(e+n/u)}\right).
\end{equation}

The probabilistic part of the proof supplies exactly the two quantities needed
for this separation. Let $\Pp$ be a product measure, let $Q\ll \Pp$, and put
$H=D(Q\Vert \Pp)$. If $H=0$, take $X=Y$. If $H>0$, we construct a coupling
$(X,Y)$ with $X\sim \Pp$ and $Y\sim Q$ such that
\begin{align}
 \sum_{i=1}^n
 \E\left[\E(\abs{X_i-Y_i}\mid X)^2\right]
 &\le C_KH\log\left(e+\frac nH\right),
 \label{eq:intro-weak-cost}\\
 \E\norm{X-Y}^2
 &\le C_K(\sqrt{nH}+H).
 \label{eq:intro-strong-cost}
\end{align}
The analogous first estimate also holds after conditioning on $Y$. The first
quantity measures the coordinatewise conditional displacement. The second is
the usual mean squared Euclidean distance. Both estimates must hold for the
same coupling because both enter one deterministic tensor inequality.

We next estimate the tensor map itself. When $x$ and the points in the support
of a probability measure $\nu$ have the same norm $r$, we prove
\begin{align*}
 \norm{x^{\otimes d}-\int y^{\otimes d}\,d\nu(y)}
 &\le \sqrt d\,r^{d-1}
 \left(\sum_i\left[\int\abs{x_i-y_i}\,d\nu(y)\right]^2\right)^{1/2}\\
 &\quad+Cd r^{d-2}\int\norm{x-y}^2\,d\nu(y).
\end{align*}
The first term comes from the linear part of the change while the norm is fixed;
the second is a quadratic remainder. Since a product subgaussian vector does
not have deterministic norm, we then allow
$R_-\le\norm x,\norm y\le R_+$, put $\Delta=R_+-R_-$, and obtain
\begin{align*}
 \norm{x^{\otimes d}-\int y^{\otimes d}\,d\nu(y)}
 &\le C\sqrt d\,R_+^{d-1}
 \left(\sum_i\left[\int\abs{x_i-y_i}\,d\nu(y)\right]^2\right)^{1/2}\\
 &\quad+CdR_+^{d-2}
 \left(R_+\Delta+\int\norm{x-y}^2\,d\nu(y)\right).
\end{align*}
A standard concentration estimate for $\norm X^2$ makes $R_+-R_-$ small with
probability $1-e^{-cu}$.

The proof of concentration then compares two level sets of
$Z=f(X^{\otimes d})$. Let $M$ be a median, keep only points for which
$\norm X$ is close to $\sqrt n$, and consider a lower set and an upper set
separated by a value gap $t$. If the upper set had probability larger than
$e^{-cu}$, the two conditional laws would have relative entropies of order
$1$ and $u$. The bounds \eqref{eq:intro-weak-cost}--\eqref{eq:intro-strong-cost}
would then make the two laws close in both quantities appearing in the tensor
estimate. On the other hand, convexity gives
\[
 f\left(\int y^{\otimes d}\,d\nu_x(y)\right)\le M,
\]
where $\nu_x$ is the conditional law of the lower point given the upper point
$x$, while $f(x^{\otimes d})\ge M+t$. The $L$-Lipschitz property therefore
forces
\[
 t\le L\norm{x^{\otimes d}-\int y^{\otimes d}\,d\nu_x(y)}.
\]
The deterministic upper bound contradicts this inequality when $t$ is a large
constant multiple of \eqref{eq:intro-scale}. This proves the main tail bound.

Both terms in \eqref{eq:intro-scale} are necessary. For the first term we use
$X$ Gaussian and $f(T)=\norm T$, so that $f(X^{\otimes d})=\norm X^d$. For the
second we construct a subgaussian distribution with rare large coordinates and
a convex Lipschitz functional that detects those coordinates. These examples
give matching lower bounds over a subgaussian class with an absolute subgaussian bound.

Euclidean functionals have additional algebraic structure. Their squares are
polynomials, so polynomial and chaos methods can use cancellations that are
not visible in the general coupling argument. We treat this direction
separately.

\section{Main results}

Throughout the paper, $n,d\in\mathbb N$ with $d\ge1$. The symbols $c,C>0$
denote absolute constants, while $c_K,C_K>0$ denote constants that depend only
on $K$; their values may change from one occurrence to the next. Unless a
subscript is displayed, $\norm{\cdot}$ denotes the Hilbert norm in the ambient
space; on $\R^n$ it is the Euclidean norm. For $p\ge1$, write
\[
 \norm Z_{L_p}=(\E\abs Z^p)^{1/p}.
\]
For $\alpha\in\{1,2\}$, define
\[
 \norm{Z}_{\psi_\alpha}
 =\inf\left\{a>0:\E\exp(\abs Z^\alpha/a^\alpha)\le2\right\}.
\]
Let $H_d=(\R^n)^{\otimes_2 d}$ be the Hilbert tensor product.

\begin{theorem}\label{thm:main}
Let $X_1,\ldots,X_n$ be independent. Assume that
\[
 \E X_i=0,\qquad \E X_i^2=1,\qquad \norm{X_i}_{\psi_2}\le K.
\]
Let $L>0$, let $f:H_d\to\R$ be convex and $L$-Lipschitz, and put
$Z=f(X^{\otimes d})$. There are constants $c_K,C_K>0$ such that
\begin{equation}\label{eq:main-tail}
 \Pp\left\{
 \abs{Z-\E Z}>
 C_KL n^{(d-1)/2}
 \left(d\sqrt u+\sqrt{d u\log\left(e+\frac nu\right)}\right)
 \right\}
 \le C e^{-c_Ku}
\end{equation}
whenever
\[
 1\le u\le c_K\frac{n}{d^2}.
\]
The corresponding moment estimate is the following. If $d\le c_K\sqrt n$, then
\begin{equation}\label{eq:main-moment}
 \norm{Z-\E Z}_{L_p}
 \le C_KL n^{(d-1)/2}
 \left(d\sqrt p+\sqrt{d p\log\left(e+\frac np\right)}\right)
\end{equation}
for
\[
 1\le p\le c_K\frac{n}{d^2}.
\]
\end{theorem}

We also state the result in a form valid for every degree. For $s>0$, define
\begin{equation}\label{eq:rate-function}
 \mathcal I_{n,d}(s)
 =\min\left\{
 \frac{s^2}{d^2},
 \frac{s^2}{d\log(e+nd/s^2)}
 \right\},
 \qquad
 \mathcal I_{n,d}(0)=0.
\end{equation}

\begin{theorem}\label{thm:full-range}
Under the assumptions of \Cref{thm:main}, for every $d\ge1$ and
\begin{equation}\label{eq:natural-range}
 0\le t\le c_KLn^{d/2},
\end{equation}
we have
\begin{equation}\label{eq:full-range-tail}
 \Pp\left\{\abs{Z-\E Z}>t\right\}
 \le C\exp\left[-c_K\mathcal I_{n,d}\left(
 \frac{t}{Ln^{(d-1)/2}}
 \right)\right].
\end{equation}
Matching minimax lower bounds under a fixed absolute subgaussian bound are
stated in \Cref{cor:rate-sharpness}.
\end{theorem}

The following theorem states the matching lower bounds.

\begin{theorem}\label{thm:lower}
There are absolute constants $K_0,c,C>0$ and an integer $n_0\ge1$ with the
following property. Let $n\ge n_0$ and let $p$ be an integer such that
\[
 1\le p\le c\frac{n}{d^2}.
\]
For each of the following two inequalities, there exist (possibly different)
independent centered, variance-one coordinates with $\psi_2$ norm at most
$K_0$ and a convex one-Lipschitz function $f:H_d\to\R$ such that:
\begin{align}
 \Pp\left\{
 \abs{f(X^{\otimes d})-m}
 \ge c d n^{(d-1)/2}\sqrt p
 \right\}
 &\ge e^{-Cp},
 \label{eq:lower-radial}\\
 \Pp\left\{
 \abs{f(X^{\otimes d})-m}
 \ge c n^{(d-1)/2}
 \sqrt{d p\log\left(e+\frac np\right)}
 \right\}
 &\ge e^{-Cp}
 \label{eq:lower-second}
\end{align}
for every $m\in\R$. Thus the right side of \eqref{eq:main-moment} is sharp up to constants even
under a fixed absolute subgaussian bound.
\end{theorem}

Here and below, minimax sharpness is understood scale by scale: for each
admissible deviation level, the marginal law and the convex Lipschitz
functional may depend on that level. A single extremizing example is not
asserted to work simultaneously at all scales.

\begin{corollary}\label{cor:rate-sharpness}
There are absolute constants $K_0,c,C>0$ such that the following holds.
Let $0<s\le c\sqrt n$ and assume
$\mathcal I_{n,d}(s)\ge1$. Then there are independent centered,
variance-one coordinates with $\psi_2$ norm at most $K_0$ and a convex
one-Lipschitz function $f$ such that, for every $m\in\R$,
\begin{equation}\label{eq:rate-sharpness}
 \Pp\left\{
 \abs{f(X^{\otimes d})-m}
 \ge c n^{(d-1)/2}s
 \right\}
 \ge\exp\big(-C\mathcal I_{n,d}(s)\big).
\end{equation}
\end{corollary}

\begin{proof}
Put
\[
 G(u)=d\sqrt u+\sqrt{du\log(e+n/u)},
 \qquad
 u_* =\sup\left\{0<u\le\frac n{d^2}:G(u)\le s\right\}.
\]
By \Cref{lem:rate-inversion},
\[
 u_*\asymp\mathcal I_{n,d}(s)\ge c.
\]
Since $G$ is increasing and continuous and
\[
 G(n/d^2)\ge\sqrt n>s
\]
after reducing $c$ in $s\le c\sqrt n$,
\[
 u_*<\frac n{d^2},\qquad G(u_*)=s.
\]
Set
\[
 p=\begin{cases}
 1,&u_*<2,\\
 \lfloor u_*\rfloor,&u_*\ge2.
 \end{cases}
\]
Then
\[
 p\asymp\mathcal I_{n,d}(s).
\]
If $p\le u_*$, then
\[
 G(p)\ge\sqrt{p/u_*}\,G(u_*)\ge cs.
\]
If $p>u_*$, then $G(p)\ge s$. Hence
\[
 \max\left\{d\sqrt p,
 \sqrt{dp\log(e+n/p)}\right\}\ge cs.
\]
Moreover,
\[
 p\le C\mathcal I_{n,d}(s)
 \le C\frac{s^2}{d^2}
 \le c\frac n{d^2}.
\]
Thus
\[
 d\frac pn\log\frac{en}{p}
 \le \frac{c}{d}\log\left(\frac{ed^2}{c}\right)
 \le Cc\log(e/c).
\]
Also, since $p\le cn$,
\[
 \log(en/p)\asymp\log(e+n/p).
\]
After reducing $c$, both lower-bound propositions are applicable. If
\[
 d\sqrt p\ge cs,
\]
use \Cref{prop:radial-lower}. Otherwise
\[
 \sqrt{dp\log(e+n/p)}\ge cs,
\]
and \Cref{prop:second-lower} applies. In either case,
\[
 \Pp\left\{
 \abs{f(X^{\otimes d})-m}\ge c n^{(d-1)/2}s
 \right\}
 \ge e^{-Cp}
 \ge\exp\big[-C\mathcal I_{n,d}(s)\big].
\]
\end{proof}

\begin{remark}
The theorem is uniform over all subgaussian marginal laws. A fixed law can
have a smaller scale. For example, a law with fixed $\norm X$ has no
fluctuation from changes in the norm. On the sphere, the natural variance scale
is $dn^{d-1}$.
\end{remark}

We shall use the following global median estimate in the proof. Put
\[
 h_n(u)=\sqrt{u\log\left(e+\frac nu\right)}.
\]

\begin{theorem}\label{thm:global}
Under the assumptions of \Cref{thm:main}, let $M$ be a median of $Z$.
There are constants $\eta_K,c_K,C_K>0$ such that the following statements hold.

If $1\le u\le \eta_Kn$, put
\[
 w_u=C_K(\sqrt{nu}+u),
 \qquad
 R_u=\sqrt{n+w_u}.
\]
Then
\begin{align}
 \Pp\Big\{
 \abs{Z-M}>C_KL\big[
 &\sqrt d\,R_u^{d-1}h_n(u)
 \notag\\
 &+dR_u^{d-2}(\sqrt{nu}+u)
 \big]\Big\}
 \le C e^{-c_Ku}.
 \label{eq:global-shell}
\end{align}
If $u\ge\max\{1,\eta_Kn\}$, put
\[
 \widetilde R_u=\sqrt n+C_K\sqrt u.
\]
Then
\begin{equation}\label{eq:global-ball}
 \Pp\left\{
 \abs{Z-M}>C_KLd\widetilde R_u^{d-1}\sqrt u
 \right\}
 \le C e^{-c_Ku}.
\end{equation}
\end{theorem}

\section{Couplings for product measures}
\label{sec:coupling}

This section builds the coupling used in the proof of the main theorem. We
first work in one dimension and then pass to product measures. We use
$a_+=\max(a,0)$ and $\log_+a=\max(\log a,0)$ for $a>0$, with
$\log_+0=0$. For probability measures $\nu\ll\mu$ we write
\[
 D(\nu\Vert\mu)=\int \log\left(\frac{d\nu}{d\mu}\right)d\nu
\]
for relative entropy.

We shall use the following elementary form of the entropy variational
inequality. If $\gamma\ll\mu$ and $\psi$ is measurable with
$\int e^\psi d\mu<\infty$, then
\begin{equation}\label{eq:entropy-variational}
 D(\gamma\Vert\mu)
 \ge \int\psi\,d\gamma-\log\int e^\psi\,d\mu.
\end{equation}
Indeed, set $d\mu_\psi=e^\psi d\mu/\int e^\psi d\mu$ and use
$D(\gamma\Vert\mu_\psi)\ge0$.

We start with a weighted moment estimate. It will be used when only a small
part of a subgaussian measure is moved by the coupling.

\begin{lemma}\label{lem:weighted-moment}
Let $\mu$ be a probability measure on $\R$. Assume that
\[
 \int e^{x^2/K^2}\,d\mu(x)\le2.
\]
Let $g\ge0$, and put
\[
 a=\int g\,d\mu,\qquad
 b=\int g\log_+g\,d\mu.
\]
If $0<a\le1$, then
\begin{equation}\label{eq:weighted-moment}
 \int x^2g(x)\,d\mu(x)
 \le K^2\left[b+a\log\left(\frac{2}{a}\right)\right].
\end{equation}
Also,
\begin{equation}\label{eq:fourth-moment}
 \int x^4\,d\mu(x)\le C K^4.
\end{equation}
\end{lemma}

\begin{proof}
Let $\gamma$ be the probability measure with density $g/a$ with respect to
$\mu$. Take \(\psi(x)=x^2/K^2\) in \eqref{eq:entropy-variational}. Since
\(\int e^{x^2/K^2}\,d\mu\le2\), we obtain
\[
 D(\gamma\Vert\mu)
 \ge \frac1{K^2}\int x^2\,d\gamma(x)-\log2.
\]
On the other hand,
\[
 D(\gamma\Vert\mu)
 =\frac1a\int g\log g\,d\mu-\log a
 \le\frac ba+\log(1/a).
\]
Multiplying by $aK^2$ proves \eqref{eq:weighted-moment}. The bound
\eqref{eq:fourth-moment} follows by integrating the tail estimate
$\mu\{\abs X>t\}\le2e^{-t^2/K^2}$.
\end{proof}

The next lemma converts the pointwise entropy density into the quantities
that occur in the maximal coupling. Let
\[
 \Phi(t)=t\log t-t+1,\qquad t\ge0,
\]
with $0\log0=0$.

\begin{lemma}\label{lem:entropy-calculus}
For $t\ge0$, put
\[
 r=(1-t)_+,
 \qquad s=(t-1)_+,
 \qquad w=\frac{s^2}{t},
 \qquad \eta=\frac{s}{t},
\]
where $w=\eta=0$ at $t=0$. Then
\begin{align}
 r^2+w+\eta^2&\le C\Phi(t),
 \label{eq:entropy-calculus-1}\\
 w\log_+w+s\log_+s&\le C\Phi(t).
 \label{eq:entropy-calculus-2}
\end{align}
\end{lemma}

\begin{proof}
If $0\le t\le1$, then $s=w=\eta=0$ and
$\Phi(t)\ge(1-t)^2/2$, so the first estimate follows and the second is
trivial. If $1\le t\le2$, then
$\Phi(t)\asymp(t-1)^2$, while each nonzero term in
\eqref{eq:entropy-calculus-1} is at most a constant times $(t-1)^2$; also
$0\le s,w\le1$, so the logarithmic terms vanish. If $t\ge2$, then
$w\le t$, $\eta\le1$, and $s\le t$, whereas
$\Phi(t)\asymp t\log t$. These bounds give both inequalities.
\end{proof}

We now construct the one-dimensional coupling. If
$\theta=d\nu/d\mu$, then $(1-\theta)_+d\mu$ is the part of $\mu$ that is missing from
$\nu$, while $(\theta-1)_+d\mu$ is the same amount of excess mass in $\nu$. The
maximal coupling leaves the common part fixed and moves only these two
remaining parts.

\begin{theorem}\label{thm:one-dimensional}
Let $\mu$ be a probability measure on $\R$ with
$\int e^{x^2/K^2}\,d\mu\le2$. Let $\nu\ll\mu$, and assume
\[
 h=D(\nu\Vert\mu)<\infty.
\]
There is a coupling $(U,V)$ of $\mu$ and $\nu$ such that
\begin{align}
 \E\left[\E(\abs{U-V}\mid U)^2\right]
 &\le C K^2h\log\left(e+\frac1h\right),
 \label{eq:one-weak-forward}\\
 \E\left[\E(\abs{U-V}\mid V)^2\right]
 &\le C K^2h\log\left(e+\frac1h\right),
 \label{eq:one-weak-reverse}\\
 \E\abs{U-V}^2
 &\le C K^2(\sqrt h+h).
 \label{eq:one-strong}
\end{align}
The right sides are interpreted as zero when $h=0$.
\end{theorem}

\begin{proof}
Let $\theta=d\nu/d\mu$ and put
\[
 r=(1-\theta)_+,
 \qquad s=(\theta-1)_+,
 \qquad
 \delta=\int r\,d\mu=\int s\,d\mu.
\]
The equality of the two integrals follows from $\int \theta\,d\mu=1$. If
$\delta=0$, then $\mu=\nu$ and we take $U=V$. Assume $\delta>0$. Let $\delta_x$ denote the probability measure concentrated at the single point $x$. Define
\begin{equation}\label{eq:maximal-coupling}
 \pi(dx,dy)
 =\min(1,\theta(x))\mu(dx)\delta_x(dy)
 +\frac1\delta r(x)\mu(dx)s(y)\mu(dy).
\end{equation}
Since $\min(1,\theta)+r=1$ and $\min(1,\theta)+s=\theta$, the two marginals of
$\pi$ are $\mu$ and $\nu$.

It is useful to normalize the two parts that are moved:
\[
 d\sigma_r=\frac r\delta\,d\mu,
 \qquad
 d\sigma_s=\frac s\delta\,d\mu.
\]
Thus the common mass is kept at $U=V$, while on the remaining mass of size
$\delta$ the variables have laws $\sigma_r$ and $\sigma_s$.

Set
\[
 R_2=\int r^2\,d\mu,
 \qquad
 W=\int \frac{s^2}{\theta}\,d\mu,
\]
where $s^2/\theta=0$ on $\{\theta=0\}$.
Since
\[
 h=\int\Phi(\theta)\,d\mu,
\]
integrating the pointwise estimates in \Cref{lem:entropy-calculus} gives
\begin{equation}\label{eq:R2-W}
 R_2+W\le Ch,
 \qquad
 R_2\le\delta,
 \qquad
 W\le\delta.
\end{equation}

We first prove \eqref{eq:one-weak-forward}. Given $U=x$, formula
\eqref{eq:maximal-coupling} says that $V=x$ on the common part and otherwise
$V$ has law $\sigma_s$. Put
\[
 m_s=\int\abs y\,d\sigma_s(y)
 =\frac1\delta\int\abs y\,s(y)\,d\mu(y).
\]
Hence
\[
 \E(\abs{U-V}\mid U=x)
 \le r(x)(\abs x+m_s),
\]
and therefore
\begin{equation}\label{eq:forward-split}
 \E\left[\E(\abs{U-V}\mid U)^2\right]
 \le2\int x^2r(x)^2\,d\mu(x)+2m_s^2R_2.
\end{equation}
Apply \Cref{lem:weighted-moment} with $g=r^2$. Since $0\le r\le1$,
its logarithmic term is zero, and
\[
 \int x^2r(x)^2\,d\mu(x)
 \le C K^2R_2\log(2/R_2)
 \le C K^2h\log(e+1/h).
\]
Next,
\[
 D(\sigma_s\Vert\mu)
 =\frac1\delta\int s\log s\,d\mu+\log(1/\delta)
 \le \frac{Ch}{\delta}+\log(1/\delta)
\]
because \(\int s\log_+s\,d\mu\le Ch\). Taking
\(\psi(x)=x^2/K^2\) in the variational inequality gives
\[
 m_s^2\le \int y^2\,d\sigma_s(y)
 \le CK^2\left(1+\frac h\delta+\log(1/\delta)\right).
\]
Together with $R_2\le\min(Ch,\delta)$,
\[
 m_s^2R_2
 \le CK^2
 \begin{cases}
 h(1+\log(1/h)),&\delta\ge h,\\
 h+\delta\log(e/\delta),&\delta<h,
 \end{cases}
 \le CK^2h\log(e+1/h).
\]
This proves \eqref{eq:one-weak-forward}.

For \eqref{eq:one-weak-reverse}, put
\[
 m_r=\int\abs x\,d\sigma_r(x).
\]
For $\nu$-almost every $y$, the conditional distribution of $U$ given
$V=y$ has a point mass at $y$ and otherwise uses $\sigma_r$. More precisely,
\[
 \E(\abs{U-V}\mid V=y)
 \le \frac{s(y)}{\theta(y)}(\abs y+m_r).
\]
Consequently,
\[
 \E\left[\E(\abs{U-V}\mid V)^2\right]
 \le2\int y^2\frac{s(y)^2}{\theta(y)}\,d\mu(y)+2m_r^2W.
\]
Apply \Cref{lem:weighted-moment} with \(g=s^2/\theta\). Here
\[
 \int\frac{s^2}{\theta}\,d\mu\le Ch,\qquad
 \int\frac{s^2}{\theta}\log_+\left(\frac{s^2}{\theta}\right)d\mu\le Ch,
\]
by \Cref{lem:entropy-calculus}. Therefore
\[
 \int y^2\frac{s(y)^2}{\theta(y)}\,d\mu(y)
 \le CK^2h\log(e+1/h).
\]
Also,
\[
 D(\sigma_r\Vert\mu)
 =\frac1\delta\int r\log r\,d\mu+\log(1/\delta)
 \le\log(1/\delta),
\]
so the variational inequality with \(\psi(x)=x^2/K^2\) gives
\[
 m_r^2\le\int x^2\,d\sigma_r(x)
 \le CK^2(1+\log(1/\delta)).
\]
Since $W\le\min(Ch,\delta)$, we distinguish two cases. If $\delta\ge h$,
then $W\le Ch$ and $\log(e/\delta)\le\log(e/h)$. If $\delta<h$, then
\[
 \delta\log(e/\delta)
 \le
 \begin{cases}
 h\log(e/h),&0<h\le1,\\
 1\le h,&h>1.
 \end{cases}
\]
Hence, in both cases,
\[
 m_r^2W\le CK^2h\log(e+1/h).
\]
This proves \eqref{eq:one-weak-reverse}.

Finally, the common part of the coupling has $U=V$. On the remaining mass,
$U$ and $V$ have laws $\sigma_r$ and $\sigma_s$, so
\begin{equation}\label{eq:strong-residual}
 \E\abs{U-V}^2
 \le2\int x^2r(x)\,d\mu(x)
 +2\int x^2s(x)\,d\mu(x).
\end{equation}
By Cauchy--Schwarz,
\begin{align*}
 \int x^2r(x)\,d\mu(x)
 &\le\left(\int x^4\,d\mu\right)^{1/2}
       \left(\int r^2\,d\mu\right)^{1/2}\\
 &\le CK^2\sqrt h.
\end{align*}
On $\{\theta>1\}$ we have the identity
\[
 s=\frac{s^2}{\theta}+\frac{s}{\theta}.
\]
For the $s/\theta$ part, Cauchy--Schwarz gives
\[
 \int x^2\frac{s}{\theta}\,d\mu
 \le\left(\int x^4\,d\mu\right)^{1/2}
      \left(\int\frac{s^2}{\theta^2}\,d\mu\right)^{1/2}
 \le CK^2\sqrt h.
\]
For the $s^2/\theta$ part, \Cref{lem:weighted-moment},
\eqref{eq:R2-W}, and \eqref{eq:entropy-calculus-2} give
\[
 \int x^2\frac{s^2}{\theta}\,d\mu\le CK^2(\sqrt h+h).
\]
Hence
\[
 \int x^2s(x)\,d\mu(x)
 \le CK^2(\sqrt h+h).
\]
Substituting the bounds for
\(\int x^2r\,d\mu\) and \(\int x^2s\,d\mu\) into
\eqref{eq:strong-residual} gives
\[
 \E|U-V|^2\le CK^2(\sqrt h+h),
\]
which is the third assertion.
\end{proof}

\medskip
\noindent\textit{From one dimension to product measures.}
Let
\[
 \Pp=\mu_1\ot\cdots\ot\mu_n,
\]
where every \(\mu_i\) has \(\psi_2\)-norm at most \(K\). For
\(x=(x_1,\ldots,x_n)\), write
\[
 x_{<i}=(x_1,\ldots,x_{i-1}),\qquad
 x_{\le i}=(x_1,\ldots,x_i),\qquad
 x_{>i}=(x_{i+1},\ldots,x_n).
\]
We use the same notation for random vectors. Thus, for example,
\(Y_{<i}=(Y_1,\ldots,Y_{i-1})\).

We shall generate pairs
\[
 (X_1,Y_1),\ldots,(X_n,Y_n)
\]
in this order. At step \(i\), the value of \(Y_{<i}\) is already known. The
conditional law used for \((X_i,Y_i)\) may depend on \(Y_{<i}\), but its first
marginal is always \(\mu_i\). Formally, the construction can be realized from
independent auxiliary random variables \(\xi_1,\ldots,\xi_n\), with
\((X_i,Y_i)=F_i(Y_{<i},\xi_i)\) for suitable measurable maps \(F_i\). The
next elementary fact records the conditioning identities that we need.

\begin{lemma}
\label{lem:full-conditioning}
Suppose that \(\xi_1,\ldots,\xi_n\) are independent and that the pairs
$(X_i,Y_i)$ are generated successively in the form
\[
 (X_i,Y_i)=F_i(Y_{<i},\xi_i),
\]
where, for every admissible value of $Y_{<i}$, the first marginal of the
resulting conditional law is $\mu_i$. Put
\[
 Z_i=\abs{X_i-Y_i},\qquad
 g_i=\E(Z_i\mid X_i,Y_{<i}),\qquad
 \widetilde g_i=\E(Z_i\mid Y_{\le i}).
\]
Then
\begin{align}
 \E(Z_i\mid X)&=\E(g_i\mid X),
 \label{eq:full-conditioning-X}\\
 \E(Z_i\mid Y)&=\widetilde g_i.
 \label{eq:full-conditioning-Y}
\end{align}
Consequently,
\begin{align}
 \E\big[\E(Z_i\mid X)^2\big]&\le\E g_i^2,
 \label{eq:full-conditioning-X-square}\\
 \E\big[\E(Z_i\mid Y)^2\big]&=\E\widetilde g_i^2.
 \label{eq:full-conditioning-Y-square}
\end{align}
\end{lemma}

\begin{proof}
Because the conditional first marginal at step $j$ is the fixed measure
$\mu_j$, the variable $X_j$ is independent of the sigma-field generated by all
previous steps. Iterating this observation shows that $X_1,\ldots,X_n$ are
independent with laws $\mu_1,\ldots,\mu_n$. Moreover, conditional on
$Y_{<i}$, the coordinates $X_{<i}$ and $X_{>i}$ provide no additional
information about the fresh choice $(X_i,Y_i)$ beyond $X_i$. Consequently,
\[
 \E(Z_i\mid X,Y_{<i})
 =\E(Z_i\mid X_i,Y_{<i})=g_i.
\]
Taking conditional expectation with respect to $X$ gives
\[
 \E(Z_i\mid X)=\E(g_i\mid X),
\]
and conditional Jensen yields
\[
 \E\big[\E(Z_i\mid X)^2\big]
 \le \E\big[\E(g_i^2\mid X)\big]
 =\E g_i^2.
\]
For the second identity, once $Y_{\le i}$ is fixed, the conditional law of the
later steps is generated from the fresh variables
$\xi_{i+1},\ldots,\xi_n$ and depends on the past only through the already fixed
$Y$-coordinates. Hence $Z_i$ and $Y_{>i}$ are conditionally independent given
$Y_{\le i}$, and therefore
\[
 \E(Z_i\mid Y)=\E(Z_i\mid Y_{\le i})=\widetilde g_i.
\]
Squaring and taking expectation gives the last assertion.
\end{proof}

For $h\ge0$, put
\[
 \Psi_n(h)=
 \begin{cases}
  h\log(e+n/h),&h>0,\\
  0,&h=0.
 \end{cases}
\]

\begin{theorem}\label{thm:product-coupling}
Let \(Q\ll \Pp\), and assume \(H=D(Q\Vert \Pp)<\infty\). There is a coupling \((X,Y)\)
with \(X\sim \Pp\) and \(Y\sim Q\) such that
\begin{align}
 \sum_{i=1}^n
 \E\left[\E(\abs{X_i-Y_i}\mid X)^2\right]
 &\le C K^2\Psi_n(H),
 \label{eq:product-weak-X}\\
 \sum_{i=1}^n
 \E\left[\E(\abs{X_i-Y_i}\mid Y)^2\right]
 &\le C K^2\Psi_n(H),
 \label{eq:product-weak-Y}\\
 \E\norm{X-Y}^2
 &\le C K^2(\sqrt{nH}+H).
 \label{eq:product-strong}
\end{align}
\end{theorem}

\begin{proof}
If $H=0$, then $Q=\Pp$, and $X=Y$ gives all three bounds. Assume $H>0$.
Let
\[
 \rho=\frac{dQ}{d\Pp}.
\]
For \(i=0,1,\ldots,n\), define
\[
 \rho_i(y_{\le i})
 =\int \rho(y_{\le i},z_{>i})
 \prod_{j=i+1}^n\mu_j(dz_j),
 \qquad \rho_0=1.
\]
Thus \(\rho_i\) is the density of the first \(i\) coordinates under \(Q\)
with respect to \(\mu_1\ot\cdots\ot\mu_i\). Let \(Q_{<i}\) denote
the law of \(Y_{<i}\) under \(Q\). Then
\[
 Q_{<i}\{\rho_{i-1}=0\}
 =\int_{\{\rho_{i-1}=0\}}\rho_{i-1}\,
 d(\mu_1\ot\cdots\ot\mu_{i-1})=0.
\]
For \(\rho_{i-1}(y_{<i})>0\), set
\begin{equation}\label{eq:conditional-density}
 \kappa_i(y_i\mid y_{<i})
 =\frac{\rho_i(y_{\le i})}{\rho_{i-1}(y_{<i})}.
\end{equation}
On the null set where the denominator is zero, set \(\kappa_i=1\). Then
\[
 \int \kappa_i(y_i\mid y_{<i})\,\mu_i(dy_i)=1
\]
for \(Q_{<i}\)-almost every \(y_{<i}\), and the conditional law of \(Y_i\)
given \(Y_{<i}=y_{<i}\) is
\[
 Q_i(dy_i\mid y_{<i})
 =\kappa_i(y_i\mid y_{<i})\mu_i(dy_i).
\]
In particular,
\[
 Q_i(\,\cdot\mid y_{<i})\ll\mu_i
 \qquad\text{for }Q_{<i}\text{-almost every }y_{<i}.
\]
Define
\[
 h_i(y_{<i})
 =\int \kappa_i(y_i\mid y_{<i})
 \log \kappa_i(y_i\mid y_{<i})\,\mu_i(dy_i).
\]
Since \(\rho_n=\rho\), the ratios in
\eqref{eq:conditional-density} give, for \(Q\)-almost every \(y\),
\[
 \rho(y)=\prod_{i=1}^n \kappa_i(y_i\mid y_{<i}).
\]
Therefore
\begin{align}
 H
 &=\E_Q\log\rho(Y)\notag\\
 &=\sum_{i=1}^n\E_Q\log \kappa_i(Y_i\mid Y_{<i})\notag\\
 &=\sum_{i=1}^n\E_Q h_i(Y_{<i}).
 \label{eq:entropy-chain}
\end{align}

We now construct \((X,Y)\). After \(Y_{<i}\) is known, apply
\Cref{thm:one-dimensional} to
\[
 \mu_i
 \quad\text{and}\quad
 \kappa_i(\,\cdot\mid Y_{<i})\mu_i.
\]
Use the resulting pair as \((X_i,Y_i)\). Formula
\eqref{eq:maximal-coupling} is explicit in \(\kappa_i\), so these choices define
one joint distribution. At every step \(X_i\) has law \(\mu_i\), independent
of the previous values; hence \(X\sim \Pp\). Also,
\[
 \mathcal L(Y_i\mid Y_{<i})
 =\kappa_i(\,\cdot\mid Y_{<i})\mu_i,
\]
so \(Y\sim Q\).

Put
\[
 \phi(t)=t\log(e+1/t),\qquad \phi(0)=0.
\]
Condition on \(Y_{<i}=y_{<i}\). Then \((X_i,Y_i)\) is the
one-dimensional coupling between \(\mu_i\) and
\(\kappa_i(\cdot\mid y_{<i})\mu_i\), and its relative entropy is
\(h_i(y_{<i})\). Hence
\[
 \E(g_i^2\mid Y_{<i})\le CK^2\phi(h_i),
 \qquad
 \E(\widetilde g_i^2\mid Y_{<i})\le CK^2\phi(h_i).
\]
The conditioning lemma then gives
\begin{align*}
 \sum_i\E\left[\E(\abs{X_i-Y_i}\mid X)^2\right]
 &\le CK^2\sum_i\E_Q\phi(h_i),\\
 \sum_i\E\left[\E(\abs{X_i-Y_i}\mid Y)^2\right]
 &\le CK^2\sum_i\E_Q\phi(h_i).
\end{align*}
A direct differentiation gives
\[
 \phi''(t)=-\frac1{t(et+1)^2}<0,
 \qquad t>0.
\]
Thus
\begin{align*}
 \sum_i\E_Q\phi(h_i)
 &\le \sum_i\phi(\E_Q h_i)\\
 &\le n\phi\left(\frac1n\sum_i\E_Qh_i\right)\\
 &=\Psi_n(H),
\end{align*}
because \(\sum_i\E_Qh_i=H\). This proves the two conditional
first-moment estimates.

For the squared distance, condition on \(Y_{<i}\). The one-dimensional
estimate reads
\[
 \E(|X_i-Y_i|^2\mid Y_{<i})
 \le CK^2(\sqrt{h_i}+h_i).
\]
Summing over \(i\) gives
\[
 \E\norm{X-Y}^2
 \le CK^2\sum_i\E_Q(\sqrt{h_i}+h_i).
\]
Moreover,
\[
 \sum_i\E_Q\sqrt{h_i}
 \le\sum_i\sqrt{\E_Qh_i}
 \le\sqrt{n\sum_i\E_Qh_i}
 =\sqrt{nH}.
\]
Together with \(\sum_i\E_Qh_i=H\), this gives
\[
 \E\norm{X-Y}^2\le CK^2(\sqrt{nH}+H),
\]
as required.
\end{proof}

The main proof uses two conditional laws of \(\Pp\). The next theorem couples
them through one copy of \(\Pp\).

\begin{theorem}\label{thm:two-set-coupling}
Let \(A,B\subset\R^n\) be measurable sets with positive \(\Pp\)-probability.
Put
\[
 Q_A=\Pp(\,\cdot\mid A),
 \qquad
 Q_B=\Pp(\,\cdot\mid B),
\]
\[
 H_A=\log\frac1{\Pp(A)},
 \qquad
 H_B=\log\frac1{\Pp(B)}.
\]
There is a coupling \((X,Y)\) with \(X\sim Q_B\) and \(Y\sim Q_A\) such
that
\begin{align}
 \sum_i\E\left[\E(\abs{X_i-Y_i}\mid X)^2\right]
 &\le CK^2\big[\Psi_n(H_A)+\Psi_n(H_B)\big],
 \label{eq:two-set-weak-X}\\
 \sum_i\E\left[\E(\abs{X_i-Y_i}\mid Y)^2\right]
 &\le CK^2\big[\Psi_n(H_A)+\Psi_n(H_B)\big],
 \label{eq:two-set-weak-Y}\\
 \E\norm{X-Y}^2
 &\le CK^2\big[
 \sqrt{nH_A}+H_A+\sqrt{nH_B}+H_B
 \big],
 \label{eq:two-set-strong}
\end{align}
\end{theorem}

\begin{proof}
Apply \Cref{thm:product-coupling} to \((\Pp,Q_A)\) and to \((\Pp,Q_B)\). Write
the two pairs as
\[
 (U,Y),\qquad U\sim \Pp,\quad Y\sim Q_A,
\]
and
\[
 (U,X),\qquad U\sim \Pp,\quad X\sim Q_B.
\]
Use the same \(U\sim \Pp\) in both pairs: after \(U\) is drawn, draw \(X\)
and \(Y\) independently according to their respective conditional laws
given \(U\). Then \(X\sim Q_B\) and \(Y\sim Q_A\).

For each \(i\),
\[
 \E(\abs{X_i-Y_i}\mid X)
 \le \E(\abs{X_i-U_i}\mid X)
 +\E(\abs{U_i-Y_i}\mid X).
\]
For the pair \((U,X)\), the product coupling theorem gives
\[
 \sum_i\E\left[\E(|X_i-U_i|\mid X)^2\right]
 \le CK^2\Psi_n(H_B).
\]
For the second term, set
\[
 c_i(U)=\E(\abs{U_i-Y_i}\mid U).
\]
Since \(X\) and \(Y\) are independent after \(U\) is fixed,
\[
 \E(\abs{U_i-Y_i}\mid X)=\E(c_i(U)\mid X),
\]
and hence
\[
 \E\left[\E(\abs{U_i-Y_i}\mid X)^2\right]
 \le \E c_i(U)^2.
\]
For the pair \((U,Y)\),
\[
 \sum_i\E c_i(U)^2\le CK^2\Psi_n(H_A).
\]
Using \((a+b)^2\le2a^2+2b^2\) and summing over \(i\) gives the first
conditional estimate. Interchanging \(X\) and \(Y\) gives the second.

Finally,
\[
 \norm{X-Y}^2
 \le2\norm{X-U}^2+2\norm{U-Y}^2.
\]
The squared-distance estimate for \((U,X)\) is at most
\(CK^2(\sqrt{nH_B}+H_B)\), and the estimate for \((U,Y)\) is at most
\(CK^2(\sqrt{nH_A}+H_A)\). Taking expectation in
$\norm{X-Y}^2\le2\norm{X-U}^2+2\norm{U-Y}^2$ gives
the stated bound for \(\E\norm{X-Y}^2\).
\end{proof}

\section{Tensor estimates for nearby norms}
\label{sec:geometry}

Write
\[
 S^{n-1}=\{x\in\mathbb R^n:\norm x=1\},
 \qquad
 B_2^n=\{x\in\mathbb R^n:\norm x\le1\}.
\]
We now estimate the change of \(x^{\otimes d}\). We first assume that the
vectors have the same norm. We then allow their norms to lie between two
nearby numbers \(R_-\) and \(R_+\).

\begin{lemma}\label{lem:differential}
For \(\Phi_d(x)=x^{\otimes d}\), with the second term in
\eqref{eq:differential-exact} omitted when $d=1$,
\begin{align}
 D\Phi_d(x)h
 &=\sum_{j=1}^d
 x^{\otimes(j-1)}\ot h\ot x^{\otimes(d-j)},
 \label{eq:differential}\\
 \norm{D\Phi_d(x)h}^2
 &=d\norm{x}^{2d-2}\norm{h}^2
 +d(d-1)\norm{x}^{2d-4}\ip{x}{h}^2.
 \label{eq:differential-exact}
\end{align}
In particular, if \(h\perp x\), then
\[
 \norm{D\Phi_d(x)h}=\sqrt d\norm{x}^{d-1}\norm h.
\]
\end{lemma}

\begin{proof}
Set
\[
 A_j=x^{\otimes(j-1)}\ot h\ot x^{\otimes(d-j)}.
\]
Then
\[
 D\Phi_d(x)h=\sum_{j=1}^d A_j,
\]
\[
 \norm{A_j}^2=\norm{x}^{2d-2}\norm h^2,
\]
and, for \(j\ne k\),
\[
 \ip{A_j}{A_k}
 =\norm{x}^{2d-4}\ip{x}{h}^2.
\]
Hence
\begin{align*}
 \norm{D\Phi_d(x)h}^2
 &=\sum_{j=1}^d\norm{A_j}^2
 +\sum_{j\ne k}\ip{A_j}{A_k}\\
 &=d\norm{x}^{2d-2}\norm h^2
 +d(d-1)\norm{x}^{2d-4}\ip{x}{h}^2.
\end{align*}
\end{proof}

We shall also use the following identity. For every \(x,y\in\R^n\),
\begin{equation}\label{eq:tensor-telescoping}
 x^{\otimes d}-y^{\otimes d}
 =\sum_{j=1}^d
 x^{\otimes(j-1)}\ot(x-y)\ot y^{\otimes(d-j)}.
\end{equation}
To verify it, put
\[
 T_j=x^{\otimes j}\ot y^{\otimes(d-j)},\qquad 0\le j\le d.
\]
Then the \(j\)th summand is \(T_j-T_{j-1}\), so the sum is
\(T_d-T_0=x^{\otimes d}-y^{\otimes d}\).

We next assume that all vectors have norm \(r\). In this case the component
of the average displacement in the direction of \(x\) is already of second
order, while the first-order part uses the factor \(\sqrt d\) from
\Cref{lem:differential}.

\begin{lemma}\label{lem:one-sphere}
Let $r>0$, let \(x\in rS^{n-1}\), and let \(\nu\) be a probability measure on
\(rS^{n-1}\). Put
\[
 q_i=\int\abs{x_i-y_i}\,d\nu(y).
\]
Then
\begin{equation}\label{eq:one-sphere}
 \norm{x^{\otimes d}-\int y^{\otimes d}\,d\nu(y)}
 \le \sqrt d\,r^{d-1}\norm{q}_2
 +Cd r^{d-2}\int\norm{x-y}^2\,d\nu(y).
\end{equation}
\end{lemma}

\begin{proof}
For $n=1$ the assertion follows directly from the two-point structure of
$rS^0$, so assume $n\ge2$. Fix \(y\in rS^{n-1}\). Choose
\(\theta\in[0,\pi]\) and a unit vector
\(u\perp x\) such that
\[
 y=(\cos\theta)x+r(\sin\theta)u.
\]
Put \(\ell=r\theta\) and
\[
 \gamma(s)=\cos(s/r)x+r\sin(s/r)u,
 \qquad 0\le s\le\ell.
\]
Then
\[
 \gamma(0)=x,\qquad \gamma(\ell)=y,
\]
\[
 \norm{\gamma(s)}=r,\qquad
 \norm{\gamma'(s)}=1,\qquad
 \ip{\gamma(s)}{\gamma'(s)}=0,
\]
and
\[
 \gamma''(s)=-\frac1{r^2}\gamma(s).
\]
Let \(F(s)=\gamma(s)^{\otimes d}\). By
\eqref{eq:differential-exact},
\[
 \norm{F'(s)}=\sqrt d\,r^{d-1}.
\]
Also,
\begin{align*}
 F''(s)
 &=\sum_{j=1}^d
 \gamma^{\otimes(j-1)}\ot\gamma''\ot
 \gamma^{\otimes(d-j)}\\
 &\quad+2\sum_{1\le j<k\le d}
 \gamma^{\otimes(j-1)}\ot\gamma'\ot
 \gamma^{\otimes(k-j-1)}\ot\gamma'\ot
 \gamma^{\otimes(d-k)}.
\end{align*}
The first sum equals \(-d\gamma(s)^{\otimes d}/r^2\). The tensors in the
second sum are mutually orthogonal because \(\gamma'(s)\perp\gamma(s)\).
Thus
\[
 \norm{F''(s)}^2
 \le d^2r^{2d-4}+4\binom d2r^{2d-4}
 \le C d^2r^{2d-4},
\]
so
\begin{equation}\label{eq:F-second}
 \norm{F''(s)}\le Cd r^{d-2}.
\end{equation}

Moreover,
\[
 \ell=r\theta\le\frac\pi2\norm{x-y},
\]
and
\[
 \ell\gamma'(0)-(y-x)
 =(1-\cos\theta)x+r(\theta-\sin\theta)u.
\]
Hence
\begin{equation}\label{eq:initial-chord}
 \norm{\ell\gamma'(0)-(y-x)}
 \le C\frac{\norm{x-y}^2}{r}.
\end{equation}
From
\[
 F(\ell)-F(0)-\ell F'(0)
 =\int_0^\ell(\ell-s)F''(s)\,ds,
\]
\eqref{eq:F-second}, \eqref{eq:initial-chord}, and
\(\norm{D\Phi_d(x)}_{\mathrm{op}}=dr^{d-1}\), we obtain
\begin{equation}\label{eq:sphere-taylor}
 \norm{\Phi_d(y)-\Phi_d(x)-D\Phi_d(x)(y-x)}
 \le Cd r^{d-2}\norm{x-y}^2.
\end{equation}

Now set
\[
 h=\int(y-x)\,d\nu(y),
\qquad
 h_{\parallel}=\frac{\ip{x}{h}}{r^2}x,
\qquad
 h_{\perp}=h-h_{\parallel}.
\]
Since
\[
 \abs{h_i}
 \le\int\abs{x_i-y_i}\,d\nu(y)=q_i,
\]
we have
\[
 \norm{h_{\perp}}\le\norm h\le\norm{q}_2
\]
and therefore
\begin{equation}\label{eq:perp-first}
 \norm{D\Phi_d(x)h_{\perp}}
 \le\sqrt d\,r^{d-1}\norm{q}_2.
\end{equation}
Because \(\norm x=\norm y=r\),
\[
 \ip{x}{y-x}=-\frac12\norm{x-y}^2,
\]
so
\[
 \ip{x}{h}
 =-\frac12\int\norm{x-y}^2\,d\nu(y)
\]
and
\begin{equation}\label{eq:parallel-first}
 \norm{D\Phi_d(x)h_{\parallel}}
 \le \frac d2r^{d-2}\int\norm{x-y}^2\,d\nu(y).
\end{equation}
Write
\[
 R(y)=\Phi_d(y)-\Phi_d(x)-D\Phi_d(x)(y-x).
\]
Then
\[
 x^{\otimes d}-\int y^{\otimes d}\,d\nu(y)
 =-D\Phi_d(x)h-\int R(y)\,d\nu(y).
\]
The component \(h_\perp\) contributes at most
\(\sqrt d\,r^{d-1}\norm{q}_2\), while \(h_\parallel\) contributes at most
\[
 \frac d2r^{d-2}\int\norm{x-y}^2\,d\nu(y).
\]
The remainder satisfies
\[
 \int\norm{R(y)}\,d\nu(y)
 \le Cd r^{d-2}\int\norm{x-y}^2\,d\nu(y).
\]
Adding these three bounds proves the lemma.
\end{proof}

A subgaussian vector does not have a fixed norm. We therefore allow
\(\norm x\) and \(\norm y\) to lie in the same short interval.

\begin{lemma}\label{lem:nearby-norms}
Let \(0<R_-\le R_+\) and \(R_+\le2R_-\). Put
\[
 \mathcal A=\{z\in\R^n:R_-\le\norm z\le R_+\},
 \qquad
 \Delta=R_+-R_-.
\]
Let \(x\in\mathcal A\), and let \(\nu\) be a probability measure on
\(\mathcal A\). Put
\[
 q_i=\int\abs{x_i-y_i}\,d\nu(y).
\]
Then
\begin{align}
 \norm{x^{\otimes d}-\int y^{\otimes d}\,d\nu(y)}
 &\le C\sqrt d\,R_+^{d-1}\norm{q}_2\notag\\
 &\quad+CdR_+^{d-2}
 \left(R_+\Delta+\int\norm{x-y}^2\,d\nu(y)\right).
 \label{eq:nearby-norms}
\end{align}
\end{lemma}

\begin{proof}
Write
\[
 r=\norm x,\qquad s=\norm y,
\qquad
 \widehat y=\frac r s y.
\]
Then
\[
 \norm{\widehat y}=r,
 \qquad
 \norm{y-\widehat y}=\abs{s-r}\le\Delta.
\]
If
\[
 \widehat q_i=\int\abs{x_i-\widehat y_i}\,d\nu(y),
\]
then
\[
 \norm{\widehat q}_2\le\norm{q}_2+\Delta
\]
and
\[
 \norm{x-\widehat y}^2
 \le C\big(\norm{x-y}^2+\Delta^2\big).
\]
Apply \Cref{lem:one-sphere} to the probability measure obtained from
\(y\mapsto\widehat y\). Since \(r\le R_+\),
\begin{align*}
 \norm{x^{\otimes d}-\int\widehat y^{\otimes d}\,d\nu(y)}
 &\le C\sqrt d\,R_+^{d-1}(\norm{q}_2+\Delta)\\
 &\quad+CdR_+^{d-2}
 \left(\int\norm{x-y}^2\,d\nu(y)+\Delta^2\right).
\end{align*}
The remaining change only changes the norm of \(y\):
\[
 \norm{\widehat y^{\otimes d}-y^{\otimes d}}
 =|r^d-s^d|\le dR_+^{d-1}\Delta.
\]
Since \(\Delta\le R_+\) and \(\sqrt d\le d\), every term containing
\(\Delta\) or \(\Delta^2\) is bounded by a constant multiple of
\[
 dR_+^{d-1}\Delta.
\]
Together with
\[
 dR_+^{d-2}\int\norm{x-y}^2\,d\nu(y),
\]
these are exactly the two quantities on the right side of
\eqref{eq:nearby-norms}.
\end{proof}

\section{Proof of the concentration bounds}

This section combines the coupling estimates from Section~\ref{sec:coupling}
with the tensor estimates from Section~\ref{sec:geometry}. We first restrict
\(X\) to a set on which \(\norm X\) is close to \(\sqrt n\). We then
compare a lower level set and an upper level set of \(f(X^{\otimes d})\).
The next section has a different purpose: it rewrites the resulting estimate
in terms of the rate function \(\mathcal I_{n,d}\) and treats bounded
coordinates by a separate argument.

We first control $\norm X$.

\begin{lemma}\label{lem:shell}
Under the assumptions of \Cref{thm:main},
\begin{equation}\label{eq:shell-tail}
 \Pp\left\{
 \abs{\norm X^2-n}>C_K(\sqrt{nu}+u)
 \right\}
 \le2e^{-c_Ku}
\end{equation}
for every $u\ge1$.
\end{lemma}

\begin{proof}
The variables \(X_i^2-1\) are independent and centered. Since
\(\norm{X_i}_{\psi_2}\le K\),
\[
 \norm{X_i^2-1}_{\psi_1}\le C_K.
\]
Hence, for every \(t>0\), Bernstein's inequality gives
\[
 \Pp\left\{\left|\sum_{i=1}^n(X_i^2-1)\right|>t\right\}
 \le2\exp\left[-c_K\min\left(\frac{t^2}{n},t\right)\right].
\]
Taking \(t=C_K(\sqrt{nu}+u)\) proves the lemma.
\end{proof}

\begin{proof}[Proof of \Cref{thm:global}]
Let $M$ be a median. Fix $C_K$ large enough for all estimates below, and then
choose $\eta_K>0$ so that
\[
 C_K(\sqrt{\eta_K}+\eta_K)\le\frac14.
\]
Assume first that $1\le u\le\eta_Kn$. The Bernstein estimate in
\Cref{lem:shell} gives
\[
 G_u=\{x:\abs{\norm x^2-n}\le w_u\},
 \qquad
 \Pp(G_u^c)\le\min\{1/4,2e^{-c_Ku}\}.
\]
Then $w_u\le n/4$, and
\[
 R_-(u)=\sqrt{n-w_u},\qquad R_+(u)=\sqrt{n+w_u},\qquad
 R_+(u)\le2R_-(u).
\]
Write $R_u=R_+(u)$.

We prove the upper-tail estimate. Put
\[
 A=G_u\cap\{Z\le M\},
 \qquad
 B=G_u\cap\{Z\ge M+t\}.
\]
The set $A$ has probability at least $1/4$. Assume that
$\Pp(B)>e^{-a u}$, where $a>0$ is a small constant. Then
\[
 H_A=\log(1/\Pp(A))\le\log4,
 \qquad
 H_B=\log(1/\Pp(B))\le au.
\]
Use \Cref{thm:two-set-coupling} with these two sets. Thus we may choose one
joint pair \((X,Y)\) such that
\[
 X\sim \Pp(\,\cdot\mid B),\qquad Y\sim \Pp(\,\cdot\mid A),
\]
and
\begin{align*}
 \sum_i\E\left[\E(|X_i-Y_i|\mid X)^2\right]
 &\le C_Ku\log(e+n/u),\\
 \E\norm{X-Y}^2&\le C_K(\sqrt{nu}+u).
\end{align*}
Here \(H_A\le\log4\), \(H_B\le au\), and \(u\ge1\).

For \(\Pp(\cdot\mid B)\)-almost every \(x\), let
\[
 \nu_x=\mathcal L(Y\mid X=x).
\]
Since \(Y\in A\),
\[
 f\left(\int y^{\otimes d}\,d\nu_x(y)\right)
 \le\int f(y^{\otimes d})\,d\nu_x(y)\le M.
\]
Since \(x\in B\), \(f(x^{\otimes d})\ge M+t\). Therefore the
\(L\)-Lipschitz property gives
\begin{equation}\label{eq:pointwise-separation}
 t\le L\norm{x^{\otimes d}-\int y^{\otimes d}\,d\nu_x(y)}.
\end{equation}

Both \(x\) and every \(y\) in the support of \(\nu_x\) lie in
\(G_u\). Hence
\[
 R_-(u)\le\norm x,\norm y\le R_+(u).
\]
Put
\[
 q_i(x)=\int|x_i-y_i|\,d\nu_x(y)
       =\E(|X_i-Y_i|\mid X=x).
\]
Then \Cref{lem:nearby-norms} gives
\begin{align*}
 \norm{x^{\otimes d}-\int y^{\otimes d}\,d\nu_x(y)}
 &\le C\sqrt d\,R_u^{d-1}\norm{q(x)}_2\\
 &\quad+CdR_u^{d-2}R_u(R_+(u)-R_-(u))\\
 &\quad+CdR_u^{d-2}\int\norm{x-y}^2\,d\nu_x(y).
\end{align*}
Taking expectation in \(x\) and using \(\E\sqrt W\le\sqrt{\E W}\),
\begin{equation}\label{eq:average-weak}
 \E\norm{q(X)}_2\le C_K\sqrt{u\log(e+n/u)}.
\end{equation}
The same coupling also satisfies
\begin{equation}\label{eq:average-strong}
 \E\norm{X-Y}^2\le C_K(\sqrt{nu}+u).
\end{equation}
Finally,
\[
 R_u(R_+(u)-R_-(u))
 =R_u\frac{2w_u}{R_+(u)+R_-(u)}
 \le C_Kw_u\le C_K(\sqrt{nu}+u).
\]
Taking expectation in \eqref{eq:pointwise-separation},
\begin{align*}
 t
 &\le CL\sqrt d\,R_u^{d-1}\E\norm{q(X)}_2\\
 &\quad+CLdR_u^{d-2}\left[
 R_u(R_+(u)-R_-(u))+\E\norm{X-Y}^2\right]\\
 &\le C_KL\left[
 \sqrt d\,R_u^{d-1}h_n(u)
 +dR_u^{d-2}(\sqrt{nu}+u)
 \right].
\end{align*}
Thus $\Pp(B)\le e^{-au}$ by the choice of $C_K$. Adding $\Pp(G_u^c)$ gives the
upper-tail bound.

For the lower tail, put
\[
 A=G_u\cap\{Z\le M-t\},
 \qquad
 B=G_u\cap\{Z\ge M\}.
\]
Now $B$ has probability at least $1/4$. If $\Pp(A)>e^{-au}$, apply
\Cref{thm:two-set-coupling} with $X\sim \Pp(\cdot\mid B)$ and
$Y\sim \Pp(\cdot\mid A)$. For $X=x$, let $\nu_x=\mathcal L(Y\mid X=x)$.
Then
\[
 f(x^{\otimes d})\ge M,\qquad
 f\left(\int y^{\otimes d}d\nu_x(y)\right)\le M-t,
\]
so
\[
 t\le L\left\|x^{\otimes d}-\int y^{\otimes d}d\nu_x(y)\right\|.
\]
With
\[
 q_i(x)=\E(|X_i-Y_i|\mid X=x),
\]
\begin{align*}
 t
 &\le CL\sqrt d\,R_u^{d-1}\E\norm{q(X)}_2\\
 &\quad+CLdR_u^{d-2}\left[
 R_u(R_+(u)-R_-(u))+\E\norm{X-Y}^2\right]\\
 &\le C_KL\left[
 \sqrt d\,R_u^{d-1}h_n(u)
 +dR_u^{d-2}(\sqrt{nu}+u)
 \right].
\end{align*}
Thus $\Pp(A)\le e^{-au}$ by the choice of $C_K$. Therefore
\[
 \Pp\{Z\le M-t\}
 \le \Pp(A)+\Pp(G_u^c)
 \le Ce^{-c_Ku}.
\]
This proves \eqref{eq:global-shell}.

Now assume that $u\ge\max\{1,\eta_Kn\}$. Let
\[
 G_u=\{x:\norm x\le\widetilde R_u\},
 \qquad
 \widetilde R_u=\sqrt n+C_K\sqrt u.
\]
The Bernstein estimate in \Cref{lem:shell} gives
\[
 \Pp(G_u^c)\le\min\{1/4,2e^{-c_Ku}\}.
\]
For every \(x,y\),
\[
 x^{\otimes d}-y^{\otimes d}
 =\sum_{j=1}^d x^{\otimes(j-1)}\ot(x-y)\ot y^{\otimes(d-j)}.
\]
Assume \(\norm x,\norm y\le\widetilde R_u\), and put
\[
 q_i=\int|x_i-y_i|\,d\nu(y).
\]
For the \(j\)th summand after integration, expand the changed slot in the
basis \((e_i)_{i=1}^n\). The terms with different \(i\) are orthogonal,
and each coefficient has norm at most \(\widetilde R_u^{d-1}q_i\). Thus
\[
 \left\|\int x^{\otimes(j-1)}\ot(x-y)\ot y^{\otimes(d-j)}
 \,d\nu(y)\right\|
 \le\widetilde R_u^{d-1}\left(\sum_iq_i^2\right)^{1/2}.
\]
Summing over \(j=1,\ldots,d\) gives
\begin{equation}\label{eq:first-order-ball}
 \norm{x^{\otimes d}-\int y^{\otimes d}\,d\nu(y)}
 \le d\widetilde R_u^{d-1}
 \left(\sum_i\left[\int\abs{x_i-y_i}\,d\nu(y)\right]^2\right)^{1/2}.
\end{equation}
For the upper tail, set
\[
 A=G_u\cap\{Z\le M\},\qquad B=G_u\cap\{Z\ge M+t\}.
\]
If $\Pp(B)>e^{-au}$, then $\Pp(A)\ge1/4$. Couple the conditional laws on $A$ and
$B$ as in \Cref{thm:two-set-coupling}. Since $H_A\le\log4$, $H_B\le au$, and
$u\ge\max\{1,\eta_Kn\}$,
\[
 \Psi_n(H_A)+\Psi_n(H_B)\le C_Ku.
\]
Hence
\[
 \E\left(\sum_i\E(|X_i-Y_i|\mid X)^2\right)^{1/2}
 \le C_K\sqrt u.
\]
For $X=x\in B$, let $\nu_x=\mathcal L(Y\mid X=x)$. Then
\[
 t\le L\left\|x^{\otimes d}-\int y^{\otimes d}\,d\nu_x(y)\right\|.
\]
Taking expectation and using \eqref{eq:first-order-ball},
\[
 t\le Ld\widetilde R_u^{d-1}
 \E\left(\sum_i\E(|X_i-Y_i|\mid X)^2\right)^{1/2}
 \le C_KLd\widetilde R_u^{d-1}\sqrt u.
\]
Thus $\Pp(B)\le e^{-au}$ by the choice of $C_K$. Adding $\Pp(G_u^c)$ gives the
upper-tail bound.

For the lower tail, put
\[
 A=G_u\cap\{Z\le M-t\},\qquad B=G_u\cap\{Z\ge M\}.
\]
If $\Pp(A)>e^{-au}$, couple $X\sim \Pp(\cdot\mid B)$ and
$Y\sim \Pp(\cdot\mid A)$. Then
\[
 \E\left(\sum_i\E(|X_i-Y_i|\mid X)^2\right)^{1/2}
 \le C_K\sqrt u,
\]
and
\[
 t\le Ld\widetilde R_u^{d-1}
 \E\left(\sum_i\E(|X_i-Y_i|\mid X)^2\right)^{1/2}
 \le C_KLd\widetilde R_u^{d-1}\sqrt u.
\]
Thus $\Pp(A)\le e^{-au}$. Therefore
\[
 \Pp\{Z\le M-t\}
 \le \Pp(A)+\Pp(G_u^c)
 \le Ce^{-c_Ku},
\]
and \eqref{eq:global-ball} follows.
\end{proof}

We now simplify the global bound in the moderate range.

\begin{corollary}\label{cor:median-tail}
Under the assumptions of \Cref{thm:main}, let $M$ be a median of $Z$. If
\[
 1\le u\le c_K\frac n{d^2},
\]
then
\begin{equation}\label{eq:median-tail}
 \Pp\left\{
 \abs{Z-M}>C_KL n^{(d-1)/2}
 \left(d\sqrt u+\sqrt{d u\log\left(e+\frac nu\right)}\right)
 \right\}
 \le Ce^{-c_Ku}.
\end{equation}
\end{corollary}

\begin{proof}
After reducing $c_K$, this range is contained in $u\le\eta_Kn$. Also,
\[
 \frac{d w_u}{n}\le C_K\left(d\sqrt{u/n}+du/n\right)\le c.
\]
Hence
\[
 R_u^{d-1}\le C n^{(d-1)/2},
 \qquad
 R_u^{d-2}(\sqrt{nu}+u)
 \le Cn^{(d-1)/2}\sqrt u.
\]
Substituting these two estimates into the moderate-range bound of
\Cref{thm:global} gives \eqref{eq:median-tail}.
\end{proof}

The next lemma converts the global tail estimate into an \(L_p\) estimate.
Taking \(p=1\) will then control the distance between the median and the
mean.

\begin{lemma}\label{lem:tail-to-moments}
Let $W\ge0$ satisfy the two bounds in \Cref{thm:global}, with $L=1$.
Then, for $1\le p\le c_Kn/d^2$,
\begin{equation}\label{eq:tail-to-moments}
 \norm{W}_{L_p}
 \le C_K n^{(d-1)/2}
 \left(d\sqrt p+\sqrt{d p\log\left(e+\frac np\right)}\right).
\end{equation}
\end{lemma}

\begin{proof}
Define
\[
 W^*(v)=\inf\{t\ge0:\Pp\{W>t\}\le v\},
 \qquad 0<v<1.
\]
Let \(V\) be uniform on \((0,1)\), and choose $C_K\ge1$ so that
\[
 U=C_K\log(e/V)\ge1.
\]
By the quantile representation,
$\norm{W}_{L_p}=\norm{W^*(V)}_{L_p}$. If
\(\Pp\{W>T(u)\}\le Ce^{-cu}\), then the definition of \(W^*\)
implies
\[
 W^*(v)\le T\big(C_K\log(e/v)\big),\qquad 0<v<1,
\]
after changing \(C_K\). Apply this observation to the two bounds in
\Cref{thm:global}.

On $\{U\le\eta_Kn\}$,
\[
 R_U\le\sqrt n+C_K\sqrt U.
\]
With $a=C_Kd/\sqrt n$,
\begin{align*}
 R_U^{d-1}
 &\le n^{(d-1)/2}e^{a\sqrt U},\\
 R_U^{d-2}(\sqrt{nU}+U)
 &\le n^{(d-1)/2}e^{a\sqrt U}
 \left(\sqrt U+\frac U{\sqrt n}\right).
\end{align*}
Hence
\begin{align*}
 W^*(V)
 &\le C_K n^{(d-1)/2}e^{a\sqrt U}
 \left[\sqrt d\,h_n(U)
 +d\left(\sqrt U+\frac U{\sqrt n}\right)\right].
\end{align*}
The variable $U$ has an exponential tail. We record the details of the
two estimates used here. Since
\[
 h_n(\lambda u)\le\sqrt{\lambda}\,h_n(u),\qquad \lambda\ge1,
\]
a dyadic decomposition of the exponential tail gives, for every $q\ge1$,
\[
 \norm{h_n(U)}_{L_q}\le C_K h_n(q)
 =C_K\sqrt{q\log(e+n/q)}.
\]
Also, the inequality
\[
 \beta\sqrt U\le \varepsilon U+\frac{\beta^2}{4\varepsilon}
\]
and the exponential tail of $U$ imply
\[
 \norm{e^{a\sqrt U}}_{L_{2p}}
 \le \exp(C_Kpa^2)\le C_K
\]
when $p\le c_Kn/d^2$. Therefore Cauchy--Schwarz gives
\[
 \norm{h_n(U)e^{a\sqrt U}}_{L_p}
 \le C_K\sqrt{p\log(e+n/p)}.
\]
Similarly, $\norm{U}_{L_q}\le C_Kq$ yields
\[
 \norm{\sqrt U+U/\sqrt n}_{L_{2p}}
 \le C_K(\sqrt p+p/\sqrt n)\le C_K\sqrt p,
\]
and another application of Cauchy--Schwarz gives
\[
 \norm{(\sqrt U+U/\sqrt n)e^{a\sqrt U}}_{L_p}
 \le C_K\sqrt p.
\]

On $\{U>\eta_Kn\}$, \eqref{eq:global-ball} gives
\[
 W^*(V)
 \le C_Kd(\sqrt n+C_K\sqrt U)^{d-1}\sqrt U
 \le C_K^d d\,U^{d/2}.
\]
Let $m=pd/2$. After reducing the constant in $p\le c_Kn/d^2$,
\[
 m\le\frac{\eta_Kn}{4C_K}.
\]
Since $\Pp\{U>u\}\le C_Ke^{-u/C_K}$,
\begin{align*}
 \E\left[U^m\1_{\{U>\eta_Kn\}}\right]
 &\le (\eta_Kn)^m\Pp\{U>\eta_Kn\}
 +m\int_{\eta_Kn}^{\infty}u^{m-1}\Pp\{U>u\}\,du\\
 &\le C_K n^m e^{-c_Kn}.
\end{align*}
Therefore
\[
 \left\|C_K^d d\,U^{d/2}\1_{\{U>\eta_Kn\}}\right\|_{L_p}
 \le C_K^d d\,n^{d/2}e^{-c_Kn/p}.
\]
Put $q=n/p$. After reducing the constant in $p\le c_Kn/d^2$,
\[
 q\ge A_Kd^2,
 \qquad
 C_K^d\sqrt q\,e^{-c_Kq}\le C_K.
\]
Therefore the bound
\(C_K^d d\,n^{d/2}e^{-c_Kn/p}\) is at most
\[
 C_Kd n^{(d-1)/2}\sqrt p.
\]
This proves \eqref{eq:tail-to-moments}.
\end{proof}

\begin{proof}[Proof of \Cref{thm:main}]
Apply \Cref{lem:tail-to-moments} to \(W=|Z-M|/L\). Then
\[
 \norm{Z-M}_{L_p}
 \le C_KL n^{(d-1)/2}
 \left(d\sqrt p+\sqrt{d p\log(e+n/p)}\right).
\]
With \(p=1\),
\[
 |\E Z-M|\le\E|Z-M|
 \le C_KL n^{(d-1)/2}\left(d+\sqrt{d\log(en)}\right).
\]
The triangle inequality therefore gives the stated bound for
\(\norm{Z-\E Z}_{L_p}\). To obtain the tail estimate, first assume
$u\ge u_0$, where $u_0>1$ is a sufficiently large absolute constant, and take
$p=c_0u$ with $c_0>0$ small enough that
$1\le p\le c_Kn/d^2$. Markov's inequality gives
\[
 \Pp\{|Z-\E Z|>e\norm{Z-\E Z}_{L_p}\}\le e^{-p}.
\]
This yields \eqref{eq:main-tail} for $u\ge u_0$, after changing constants. For
$1\le u<u_0$, enlarge the absolute prefactor $C$ in
\eqref{eq:main-tail} so that the bound is trivial. Thus
\eqref{eq:main-tail} holds throughout the stated range.
\end{proof}

\section{The full natural range and bounded coordinates}

The proof of
\Cref{thm:main} gives the concentration estimate with the probability parameter
\(u\). We now make two further deductions. First, we solve for \(u\) in terms
of the deviation size and obtain the rate function \(\mathcal I_{n,d}\).
Second, under the stronger assumption \(|X_i|\le K\), we give a different
proof based on Talagrand's product inequality. The bounded-coordinate proof
does not use the subgaussian coupling from Section~\ref{sec:coupling}; this
is why it is kept separate from the proof of the main theorem.

\begin{lemma}\label{lem:rate-inversion}
Put
\[
 G_{n,d}(u)=d\sqrt u+
 \sqrt{du\log(e+n/u)},\qquad u>0.
\]
There are absolute constants $c,C>0$ such that, for $0<s\le c\sqrt n$,
\begin{equation}\label{eq:rate-inversion}
 c\mathcal I_{n,d}(s)
 \le\sup\left\{0<u\le\frac n{d^2}:G_{n,d}(u)\le s\right\}
 \le C\mathcal I_{n,d}(s).
\end{equation}
The supremum of the empty set is understood as zero.
\end{lemma}

\begin{proof}
Set
\[
 a=\frac{s^2}{d^2},\qquad
 v=\frac{s^2}{d},\qquad
 L_v=\log(e+n/v),\qquad
 b=\frac{v}{L_v}.
\]
Suppose first that \(G_{n,d}(u)\le s\). Then
\[
 d\sqrt u\le s,
\]
so \(u\le a\). The second term gives
\[
 u\log(e+n/u)\le v.
\]
Put
\[
 F(u)=u\log(e+n/u).
\]
The function \(F\) is increasing, and
\[
 F(b)=\frac{v}{L_v}\log\left(e+\frac nvL_v\right)\ge v,
\]
because \(L_v\ge1\). Thus \(F(u)\le v\) implies \(u\le b\). Hence
\[
 u\le\min(a,b)=\mathcal I_{n,d}(s).
\]
This proves the upper estimate in \eqref{eq:rate-inversion}.

For the lower estimate, take
\[
 u=c_0\min(a,b)
\]
with \(c_0>0\) sufficiently small. Then
\[
 d\sqrt u\le\sqrt{c_0}\,s.
\]
It remains to control the second term in $G_{n,d}(u)$.

Suppose first that $b\le a$. Then $u=c_0b=c_0v/L_v$, and
\[
 \log(e+n/u)
 =\log\left(e+\frac{nL_v}{c_0v}\right)
 \le C L_v.
\]
Consequently,
\[
 du\log(e+n/u)
 \le Cc_0\,d\frac{v}{L_v}L_v
 =Cc_0s^2.
\]

Now suppose that $a<b$. Since
\[
 \frac{s^2}{d^2}<\frac{s^2/d}{L_v},
\]
we have $L_v<d$. In this case $u=c_0a=c_0s^2/d^2=c_0v/d$, and hence
\begin{align*}
 \log(e+n/u)
 &=\log\left(e+\frac{nd}{c_0v}\right)\\
 &\le \log(e+n/v)+\log(d/c_0)\\
 &\le C d.
\end{align*}
Therefore
\[
 du\log(e+n/u)
 \le d\,\frac{c_0v}{d}\,Cd
 =Cc_0s^2.
\]
Thus, in both cases,
\[
 \sqrt{du\log(e+n/u)}\le C\sqrt{c_0}\,s.
\]
For sufficiently small $c_0$, the two terms in $G_{n,d}(u)$ have sum at most
$s$. Finally, $s\le c\sqrt n$ ensures $u\le n/d^2$ after reducing $c$.
Thus the supremum in \eqref{eq:rate-inversion} is at least
$c\mathcal I_{n,d}(s)$.
\end{proof}

\begin{proof}[Proof of \Cref{thm:full-range}]
Put
\[
 s=\frac{t}{Ln^{(d-1)/2}}.
\]
Fix constants $C_{K,0}>0$ and $\gamma_K>0$ for which
\eqref{eq:main-tail} holds with threshold
$C_{K,0}G_{n,d}(u)$ on $1\le u\le\gamma_Kn/d^2$.
Assume first that $\mathcal I_{n,d}(s)$ is larger than a sufficiently
large constant depending only on $K$. Increase $C_{K,0}$, if necessary, so that
$2C_{K,0}\ge1$, and put
\[
 s_0=\frac{s}{2C_{K,0}},
 \qquad
 a=2C_{K,0}.
\]
Then $a\ge1$ and $s_0=s/a$. Also,
\[
 \frac{s_0^2}{d^2}=a^{-2}\frac{s^2}{d^2},
\]
and, with $x=nd/s^2$,
\[
 \log(e+x)\le\log(e+a^2x)
 \le\log(e+x)+2\log a
 \le(1+2\log a)\log(e+x).
\]
Hence
\[
 \mathcal I_{n,d}(s_0)\asymp_K\mathcal I_{n,d}(s).
\]
After reducing the constant in the natural range, $s_0$ is in the range of
\Cref{lem:rate-inversion}. Put
\[
 u_0=\sup\left\{0<u\le\frac n{d^2}:G_{n,d}(u)\le s_0\right\}.
\]
Then
\[
 u_0\asymp_K\mathcal I_{n,d}(s),
 \qquad
 G_{n,d}(u_0)\le s_0.
\]
Decrease $\gamma_K$, if necessary, so that $\gamma_K\le1$, and put
\[
 u=\gamma_Ku_0.
\]
Choose the $K$-dependent lower threshold for $\mathcal I_{n,d}(s)$ large
enough that $u\ge1$. Then
\[
 u\le\gamma_K\frac n{d^2},
 \qquad
 u\asymp_K\mathcal I_{n,d}(s),
\]
and, since $G_{n,d}$ is increasing,
\[
 C_{K,0}G_{n,d}(u)
 \le C_{K,0}G_{n,d}(u_0)
 \le\frac s2.
\]
Hence \eqref{eq:main-tail} gives
\[
 \Pp\{\abs{Z-\E Z}>t\}
 \le C\exp[-c_K\mathcal I_{n,d}(s)].
\]

If $\mathcal I_{n,d}(s)$ is below the $K$-dependent threshold in the first
case,
decrease the constant $c_K$ in the exponent, if necessary, so that the product
of $c_K$ with that threshold is at most one. Then the absolute choice $C=e$
ensures
\[
 C\exp[-c_K\mathcal I_{n,d}(s)]\ge1,
\]
and the required probability bound is automatic. The assumption
$t\le c_KLn^{d/2}$ also gives $s\le c_K\sqrt n$.
\end{proof}

We next treat bounded coordinates. For simple tensors with independent
factors, Vershynin's convex theorem has variance scale $d n^{d-1}$
\cite[Theorem~1.3]{Vershynin}. In the symmetric model, the radial example
$f(T)=\norm T$ forces the larger scale $d^2n^{d-1}$. The next theorem shows
that this larger scale is sufficient and that no logarithmic factor is needed.

\begin{theorem}
\label{thm:bounded-full}
Let $X_1,\ldots,X_n$ be independent and assume
\[
 \E X_i=0,\qquad \E X_i^2=1,\qquad \abs{X_i}\le K.
\]
Let $L>0$, let $f:H_d\to\R$ be convex and $L$-Lipschitz, and put
$Z=f(X^{\otimes d})$. Then, for every $d\ge1$ and
$0\le t\le c_KLn^{d/2}$,
\begin{equation}\label{eq:bounded-full}
 \Pp\{\abs{Z-\E Z}>t\}
 \le C\exp\left(-\frac{c_Kt^2}{L^2d^2n^{d-1}}\right).
\end{equation}
The variance scale $d^2n^{d-1}$ is sharp for nontrivial moderate deviations.
This already holds for one fixed bounded marginal law with an absolute support
bound.
\end{theorem}

\begin{proof}
Let $\mu_i$ be the law of $X_i$, and put
\[
 \Omega_i=[-K,K],\qquad \Omega=\prod_{i=1}^n\Omega_i,
 \qquad \Pp=\mu_1\ot\cdots\ot\mu_n.
\]
Let $M$ be a median. We use Talagrand's convex-distance inequality
\cite[Section~4.1, Theorem~4.1.1]{Talagrand} on $\Omega$. For $x,y\in\Omega$,
associate with $y$ the vector
\[
 \big(\1_{\{x_1\ne y_1\}},\ldots,\1_{\{x_n\ne y_n\}}\big)\in\{0,1\}^n.
\]
For $A\subset\Omega$,
\[
 d_T(x,A)=\inf_{\nu\in\cP_f(A)}
 \left(\sum_i\nu\{y:y_i\ne x_i\}^2\right)^{1/2},
\]
where $\cP_f(A)$ denotes the probability measures on $A$ with finite support.
Talagrand's theorem states that
\[
 \Pp(A)\Pp\{d_T(X,A)\ge r\}\le e^{-r^2/4}.
\]

Fix \(R>0\), let $x\in\Omega\cap RB_2^n$, and let $\nu$ be supported on
$\Omega\cap RB_2^n$. Starting from \eqref{eq:tensor-telescoping}, write the integrated
\(j\)th summand as
\[
 T_j=\int x^{\otimes(j-1)}\ot(x-y)\ot y^{\otimes(d-j)}\,d\nu(y).
\]
Expanding \(x-y=\sum_i(x_i-y_i)e_i\), the terms with different \(i\) are
orthogonal in the \(j\)th tensor slot. Hence
\begin{align*}
 \norm{T_j}^2
 &\le R^{2d-2}\sum_i
 \left(\int|x_i-y_i|\,d\nu(y)\right)^2\\
 &\le4K^2R^{2d-2}\sum_i\nu\{y:y_i\ne x_i\}^2,
\end{align*}
because \(|x_i-y_i|\le2K\) whenever \(x_i\ne y_i\). Summing over
\(j=1,\ldots,d\) gives
\begin{equation}\label{eq:bounded-mixture}
 \norm{x^{\otimes d}-\int y^{\otimes d}\,d\nu(y)}
 \le2KdR^{d-1}
 \left(\sum_i\nu\{y:y_i\ne x_i\}^2\right)^{1/2}.
\end{equation}

For every tensor \(T\), convexity and the \(L\)-Lipschitz property imply
that there is a vector \(v\) with \(\norm v\le L\) such that
\[
 f(S)\ge f(T)+\ip{v}{S-T}
 \qquad\text{for every }S.
\]
Consider, for example, $x\in\Omega$ with \(\norm x\le R\) and
\(f(x^{\otimes d})\ge M+t\). Let
\[
 A_R=\{y\in\Omega:\norm y\le R,\ f(y^{\otimes d})\le M\}.
\]
For every probability measure \(\nu\) with finite support in \(A_R\),
convexity at \(T=x^{\otimes d}\) gives
\begin{align*}
 t
 &\le f(x^{\otimes d})-\int f(y^{\otimes d})\,d\nu(y)\\
 &\le\ip{v}{x^{\otimes d}-\int y^{\otimes d}\,d\nu(y)}\\
 &\le2KLdR^{d-1}
 \left(\sum_i\nu\{y:y_i\ne x_i\}^2\right)^{1/2}.
\end{align*}
Taking the infimum over such \(\nu\) gives
\[
 d_T(x,A_R)\ge\frac{t}{2KLdR^{d-1}}
 \qquad\text{whenever }\norm x\le R\text{ and }f(x^{\otimes d})\ge M+t.
\]
If $\Pp\{\norm X>R\}<1/4$, then
\[
 \Pp(A_R)\ge\frac14.
\]
Therefore
\[
 \Pp\{\norm X\le R,\ Z\ge M+t\}
 \le4\exp\left(-\frac{t^2}{16K^2L^2d^2R^{2d-2}}\right).
\]
For the lower tail, put
\begin{align*}
 A_R^-&=\{y\in\Omega:\norm y\le R,\ f(y^{\otimes d})\le M-t\},\\
 B_R&=\{x\in\Omega:\norm x\le R,\ f(x^{\otimes d})\ge M\}.
\end{align*}
If $x\in B_R$ and $\nu$ is supported on $A_R^-$, then
\[
 t\le2KLdR^{d-1}
 \left(\sum_i\nu\{y:y_i\ne x_i\}^2\right)^{1/2},
\]
so
\[
 d_T(x,A_R^-)\ge\frac{t}{2KLdR^{d-1}}.
\]
Also $\Pp(B_R)\ge1/4$. Hence
\[
 \Pp(A_R^-)
 \le4\exp\left(-\frac{t^2}{16K^2L^2d^2R^{2d-2}}\right).
\]
If $\Pp\{\norm X>R\}\ge1/4$, then $4\Pp\{\norm X>R\}\ge1$. Thus
\begin{equation}\label{eq:bounded-local}
 \Pp\{\abs{Z-M}>t\}
 \le4\Pp\{\norm X>R\}
 +8\exp\left(-\frac{c_Kt^2}{L^2d^2R^{2d-2}}\right).
\end{equation}

Since \(|X_i|\le K\), the variables \(X_i^2-1\) are centered and
uniformly bounded. Bernstein's inequality therefore gives, for every
\(u\ge1\),
\[
 \Pp\left\{\norm X^2>
 n+C_K(\sqrt n\,u+u^2)\right\}
 \le2e^{-c_Ku^2}.
\]
Choose this radius in \eqref{eq:bounded-local}. If
\[
 1\le u\le c_K\frac{\sqrt n}{d},
\]
then, with this choice of $R$,
\[
 \left(\frac{R}{\sqrt n}\right)^{d-1}
 \le \exp\left[C_Kd\left(\frac{u}{\sqrt n}
              +\frac{u^2}{n}\right)\right]
 \le C_K.
\]
Thus $R^{d-1}\le C_Kn^{(d-1)/2}$ throughout this range, and hence
\[
 \Pp\left\{\abs{Z-M}>
 C_KLd n^{(d-1)/2}u\right\}
 \le Ce^{-c_Ku^2}.
\]
For $u\ge1$, put
\[
 R=\sqrt n+C_Ku.
\]
Bernstein's inequality gives
\[
 \Pp\{\norm X>R\}\le2e^{-c_Ku^2}.
\]
Substituting this $R$ and
\[
 t=C_KLd\,u(\sqrt n+C_Ku)^{d-1}
\]
in \eqref{eq:bounded-local} gives
\[
 \Pp\left\{\abs{Z-M}>
 C_KLd\,u(\sqrt n+C_Ku)^{d-1}\right\}
 \le Ce^{-c_Ku^2},\qquad u\ge1.
\]
Let $V$ be uniform on $(0,1)$ and put
$U=C_K\sqrt{\log(e/V)}$. If
\[
 R^*(v)=\inf\{t\ge0:\Pp\{\abs{Z-M}>t\}\le v\},
\]
then $R^*(V)$ has the same distribution as $\abs{Z-M}$. After increasing
$C_K$ in the definition of $U$,
\[
 R^*(V)\le C_KLd\,U(\sqrt n+C_KU)^{d-1}.
\]
If $d\le c_K\sqrt n$, then
\[
 (\sqrt n+C_KU)^{d-1}
 \le n^{(d-1)/2}\exp(C_KdU/\sqrt n).
\]
The random variable \(U=C_K\sqrt{\log(e/V)}\) has a subgaussian tail. If
\(d\le c_K\sqrt n\), its exponential moment satisfies
\[
 \E\left[U\exp(C_KdU/\sqrt n)\right]\le C_K.
\]
Therefore
\[
 \abs{\E Z-M}
 \le\E\abs{Z-M}
 \le C_KLd n^{(d-1)/2}.
\]
For
\[
 1\le u\le c_K\frac{\sqrt n}{d},
\]
after increasing $C_K$,
\[
 \Pp\left\{\abs{Z-\E Z}>
 C_KLd n^{(d-1)/2}u\right\}
 \le Ce^{-c_Ku^2}.
\]
Set
\[
 u=\frac{t}{C_KLd n^{(d-1)/2}}.
\]
After reducing the constant in the natural range,
\[
 u\le c_K\frac{\sqrt n}{d}.
\]
If $u\ge1$, the centered estimate gives \eqref{eq:bounded-full}. If
$u<1$, then $t^2/(L^2d^2n^{d-1})$ is bounded by a constant depending only
on $K$. Reduce the constant in the exponent and increase $C$ so that the
right side of \eqref{eq:bounded-full} is at least one.

It remains to justify the sharpness statement. Take, for example, the fixed
law
\[
 \Pp\{X_i=0\}=\frac12,\qquad
 \Pp\{X_i=\sqrt2\}=\Pp\{X_i=-\sqrt2\}=\frac14.
\]
It is centered, has variance one, is bounded by $\sqrt2$, and satisfies
$\Var(X_i^2)=1$. By the central limit theorem, there are absolute constants
$a,c_0>0$ such that, for all sufficiently large $n$,
\[
 \Pp\left\{\norm X^2\ge n+a\sqrt n\right\}\ge c_0,
 \qquad
 \Pp\left\{\norm X^2\le n-a\sqrt n\right\}\ge c_0.
\]
For $d\le c\sqrt n$, the mean-value theorem gives
\[
 (n+a\sqrt n)^{d/2}-(n-a\sqrt n)^{d/2}
 \ge c d n^{(d-1)/2}.
\]
Taking $f(T)=\norm T$, the two events above therefore place
$f(X^{\otimes d})=\norm X^d$ in two sets separated by
$c d n^{(d-1)/2}$, each with probability at least $c_0$. For every center
$m\in\R$, one of the two sets is at distance at least half this amount from
$m$. This proves sharpness of the variance scale in the nontrivial
moderate-deviation regime; for larger $d$ the natural range contains no
nontrivial fixed-exponent moderate regime.
\end{proof}

\begin{corollary}\label{cor:high-degree}
Under the assumptions of \Cref{thm:main}, let
\[
 1\le p\le c_K\frac n{d^2}.
\]
If
\[
 d\ge\log(e+n/p),
\]
then
\[
 \norm{Z-\E Z}_{L_p}
 \le C_KLd n^{(d-1)/2}\sqrt p.
\]
This order is sharp over a subgaussian class with an absolute subgaussian
bound, even for the Euclidean functional $f(T)=\norm T$.
\end{corollary}

\begin{proof}
Since $p\ge1$ and $p\le c_Kn/d^2$,
\[
 d^2\le c_Kn.
\]
After reducing the constant in the corollary, the degree condition in
\Cref{thm:main} holds. Also,
\[
 \sqrt{dp\log(e+n/p)}\le d\sqrt p
\]
when $d\ge\log(e+n/p)$. Thus \eqref{eq:main-moment} gives the upper
bound.

Reduce the constant in the corollary further, if necessary, so that the
range in \Cref{prop:radial-lower} also holds. For sharpness, take
$X=G\sim N(0,I_n)$ and $f(T)=\norm T$. With $m=\E\norm G^d$, \Cref{prop:radial-lower} gives
\[
 \Pp\left\{
 \abs{\norm G^d-\E\norm G^d}
 \ge c d n^{(d-1)/2}\sqrt p
 \right\}\ge e^{-Cp}.
\]
Therefore
\[
 \norm{\norm G^d-\E\norm G^d}_{L_p}
 \ge c d n^{(d-1)/2}\sqrt p\,e^{-C}
 \ge c d n^{(d-1)/2}\sqrt p.
\]
\end{proof}

\begin{corollary}\label{cor:euclidean}
Let $A:H_d\to\mathcal H$ be a bounded operator into a Hilbert space and put
\[
 Y=\norm{AX^{\otimes d}}_{\mathcal H},
 \qquad
 \sigma_A=\big(\E Y^2\big)^{1/2}.
\]
If $A\ne0$, then \Cref{thm:full-range} applies with
$L=\norm{A}_{\mathrm{op}}$; if $A=0$, all conclusions below are trivial. If
$d\le c_K\sqrt n$, then
\begin{equation}\label{eq:rms-shift}
 \abs{\sigma_A-\E Y}
 \le C_K\norm{A}_{\mathrm{op}}n^{(d-1)/2}
 \left(d+\sqrt{d\log(en)}\right).
\end{equation}
Also,
\[
 \sigma_A^2=\operatorname{Tr}(A^*A\Sigma_d),
\]
where $\Sigma_d:H_d\to H_d$ is the second-moment operator
\[
 \Sigma_d T
 =\E\left[\ip{X^{\otimes d}}{T}X^{\otimes d}\right],
 \qquad T\in H_d.
\]
\end{corollary}

\begin{proof}
The map $T\mapsto\norm{AT}_{\mathcal H}$ is convex and
$\norm{A}_{\mathrm{op}}$-Lipschitz. The identity for $\sigma_A$ follows by
expanding the Hilbert norm. Finally,
\[
 0\le\sigma_A-\E Y
 \le\big(\Var Y\big)^{1/2}
 =\norm{Y-\E Y}_{L_2}.
\]
The moment estimate in \Cref{thm:main}, taken at $p=2$, gives \eqref{eq:rms-shift}.
\end{proof}

\section{Sharpness and optimality}

We prove the two parts of \Cref{thm:lower} separately.

\medskip
\noindent\textit{Lower bound for the first term.}

\begin{proposition}\label{prop:radial-lower}
Let $G\sim N(0,I_n)$ and put $f(T)=\norm T$. If
$1\le p\le cn/d^2$, then, for every $m\in\R$,
\[
 \Pp\left\{
 \abs{f(G^{\otimes d})-m}
 \ge cd n^{(d-1)/2}\sqrt p
 \right\}
 \ge e^{-Cp}.
\]
\end{proposition}

\begin{proof}
The range $1\le p\le cn/d^2$ is nonempty only if
\[
 n\ge c^{-1}d^2.
\]
After reducing $c$, we may assume $n\ge4$. Let $R=\norm G$. Its density is
\[
 \varphi_n(r)=\frac{2^{1-n/2}}{\Gamma(n/2)}r^{n-1}e^{-r^2/2}.
\]
Its mode is \(r_0=\sqrt{n-1}\). Stirling's formula gives
\(\varphi_n(r_0)\asymp1\). There is an absolute $c_0>0$ such that,
for $0\le s\le c_0\sqrt n$,
\begin{align*}
 \log\frac{\varphi_n(r_0+s)}{\varphi_n(r_0)}
 &=(n-1)\log\left(1+\frac{s}{r_0}\right)-r_0s-\frac{s^2}{2}
 \ge-Cs^2,\\
 \log\frac{\varphi_n(r_0-s)}{\varphi_n(r_0)}
 &=(n-1)\log\left(1-\frac{s}{r_0}\right)+r_0s-\frac{s^2}{2}
 \ge-Cs^2.
\end{align*}
Reduce the constant $c$ in the proposition so that, for
$1\le p\le cn$,
\[
 \sqrt p+1\le c_0\sqrt n.
\]
Integrating over $[r_0+\sqrt p,r_0+\sqrt p+1]$ and
$[r_0-\sqrt p-1,r_0-\sqrt p]$ gives
\[
 \Pp\{R\ge r_0+\sqrt p\}\ge e^{-Cp},
 \qquad
 \Pp\{R\le r_0-\sqrt p\}\ge e^{-Cp}.
\]
Moreover,
\[
 \frac{r_0-\sqrt p}{\sqrt n}
 \ge 1-\frac{C}{n}-\sqrt{\frac pn}
 \ge 1-\frac{C\sqrt c}{d}.
\]
Reduce $c$ further so that $C\sqrt c\le1/2$. Then
\[
 \left(1-\frac{C\sqrt c}{d}\right)^{d-1}
 \ge e^{-2C\sqrt c}\ge e^{-1},
\]
and hence
\[
 (r_0-\sqrt p)^{d-1}\ge c n^{(d-1)/2}.
\]
Therefore
\[
 (r_0+\sqrt p)^d-(r_0-\sqrt p)^d
 \ge cd n^{(d-1)/2}\sqrt p.
\]
Since $f(G^{\otimes d})=R^d$, one of the two events is at distance at least
half this gap from any fixed $m$.
\end{proof}

\medskip
\noindent\textit{Lower bound for the second term.}

\begin{proposition}\label{prop:second-lower}
There are absolute constants $K_0,c,C>0$ such that the following holds.
Let $p$ be an integer such that
\[
 1\le p\le cn,
 \qquad
 dp\log(en/p)\le cn.
\]
There are independent centered variance-one random variables $X_i$ such that
\[
 \norm{X_i}_{\psi_2}\le K_0\qquad(1\le i\le n),
\]
and there is a convex one-Lipschitz function $f$ such that, for every
$m\in\R$,
\[
 \Pp\left\{
 \abs{f(X^{\otimes d})-m}
 \ge c n^{(d-1)/2}\sqrt{dp\log(en/p)}
 \right\}
 \ge e^{-Cp}.
\]
\end{proposition}

\begin{proof}
Put
\[
 \rho=\frac pn,
 \qquad
 a^2=16\log(en/p),
 \qquad
 b^2=\frac{1-\rho a^2}{1-\rho}.
\]
Since $\rho=p/n$ and $a^2=16\log(en/p)$,
\[
 1-b^2=\frac{\rho(a^2-1)}{1-\rho}
 \le C\frac pn\log(en/p).
\]
After reducing the absolute constant in the assumptions, $b^2\ge3/4$.
Moreover, $dp\log(en/p)\le cn$ implies
$d(1-b^2)\le Cc$, and therefore
\[
 b^{d-1}\ge\exp[-C d(1-b^2)]\ge c.
\]
Let the variables be independent with
\[
 \Pp\{X_i=a\}=\Pp\{X_i=-a\}=\frac\rho2,
 \qquad
 \Pp\{X_i=b\}=\Pp\{X_i=-b\}=\frac{1-\rho}{2}.
\]
They are centered and have variance one by the definition of $b$. Choose
$K_0$ so that
\[
 \theta:=\frac{16}{K_0^2}\le\frac12,
 \qquad
 e^{1/K_0^2}\le\frac32.
\]
Since $\rho=p/n\le c$, reduce $c$ so that
\[
 e^\theta c^{1-\theta}\le\frac12.
\]
As $b^2\le1$,
\begin{align*}
 \E e^{X_i^2/K_0^2}
 &\le e^{1/K_0^2}+\rho\left(\frac e\rho\right)^\theta\\
 &=e^{1/K_0^2}+e^\theta\rho^{1-\theta}\\
 &\le2.
\end{align*}
Thus $\norm{X_i}_{\psi_2}\le K_0$.

For $S\subset\{1,\ldots,n\}$, let $V_S$ be the coordinate subspace
spanned by $(e_i)_{i\in S}$. Let $E_S\subset H_d$ be the orthogonal sum of
the $d$ tensor spaces in which exactly one slot lies in $V_S$ and all other
slots lie in $V_S^\perp$. Let $P_S$ be the orthogonal projection onto
$E_S$. Define
\[
 f(T)=\max_{\abs S=p}\norm{P_ST}.
\]
This function is convex and one-Lipschitz.

Write $x=x_S+x_{S^c}$. The $d$ parts of $P_Sx^{\otimes d}$ are orthogonal.
Thus
\begin{equation}\label{eq:one-slot-projection}
 \norm{P_Sx^{\otimes d}}
 =\sqrt d\norm{x_S}\norm{x_{S^c}}^{d-1}.
\end{equation}
Let $N$ be the number of coordinates with absolute value $a$. On $N=0$,
\[
 f(X^{\otimes d})
 =\sqrt d\,b^d\sqrt p\,(n-p)^{(d-1)/2}=:z_0.
\]
On $N\ge p/2$, choose $S$ of size $p$ containing at least $p/2$ coordinates
with absolute value $a$. For this set $S$, the identity
\[
 \norm{P_Sx^{\otimes d}}
 =\sqrt d\,\norm{x_S}\norm{x_{S^c}}^{d-1}
\]
gives
\[
 f(X^{\otimes d})
 \ge\sqrt d\,a\sqrt{p/2}\,b^{d-1}(n-p)^{(d-1)/2}=:z_1.
\]
The assumption $dp\log(en/p)\le cn$ gives
\[
 d\frac pn\le c.
\]
Since $p/n\le c\le1/2$,
\[
 \left(1-\frac pn\right)^{(d-1)/2}
 \ge \exp\left(-C d\frac pn\right)\ge c.
\]
Also $a\ge4$ and $b\le1$, so
\[
 \frac{a}{\sqrt2}-b\ge ca.
\]
Thus
\begin{align*}
 z_1-z_0
 &=\sqrt{dp}\,b^{d-1}(n-p)^{(d-1)/2}
 \left(\frac a{\sqrt2}-b\right)\\
 &\ge c n^{(d-1)/2}\sqrt{dp\log(en/p)}.
\end{align*}
The variable \(N\) is binomial with mean \(p\). Since $p/n\le c\le1/2$,
\[
 \Pp\{N=0\}=(1-p/n)^n\ge e^{-2p}.
\]
Also
\[
 \E N^2=p^2+p(1-p/n)\le p^2+p.
\]
Therefore
\[
 \Pp\{N\ge p/2\}
 \ge\frac{(\E N-p/2)^2}{\E N^2}
 \ge\frac18.
\]
One of these two events is at
distance at least $(z_1-z_0)/2$ from any fixed $m$.
\end{proof}

\begin{proof}[Proof of \Cref{thm:lower}]
The first claim is \Cref{prop:radial-lower}. The second claim is
\Cref{prop:second-lower}. In the range $p\le cn$,
$\log(en/p)\asymp\log(e+n/p)$, so the logarithmic conventions in the
proposition and in \eqref{eq:lower-second} are equivalent up to absolute
constants. If $p\le cn/d^2$, then
\[
 dp\log(en/p)\le cn
\]
after reducing $c$. Indeed, the function
$x\mapsto x\log(e/x)$ is increasing on $(0,1]$, so, with $x=p/n$,
\[
 d\frac pn\log\frac{en}{p}
 \le \frac{c}{d}\log\left(\frac{e d^2}{c}\right)
 \le Cc\log(e/c).
\]
Since $c\log(e/c)\to0$ as $c\downarrow0$, \Cref{prop:second-lower} applies
after reducing $c$. The maximum of the two lower scales is comparable with
their sum.
\end{proof}

\medskip
\noindent\textit{Why the mean squared distance bound cannot be improved.}

For $x,y\in RB_2^n$, expand the changed tensor slot in
\eqref{eq:tensor-telescoping} in the standard basis. Terms with different
coordinate indices are orthogonal. Hence
\[
 \norm{x^{\otimes d}-\int y^{\otimes d}\,d\nu(y)}
 \le dR^{d-1}
 \left(\sum_i\left[\int\abs{x_i-y_i}\,d\nu(y)\right]^2\right)^{1/2}.
\]
It gives the correct factor for changes in the norm. It gives the wrong factor for directions orthogonal to $x$.
Replacing $d$ by $\sqrt d$ in this inequality is false, even on a sphere.
The missing contribution is
therefore contained in the quadratic remainder.

For vectors with $R_-\le\norm x,\norm y\le R_+$, \Cref{lem:nearby-norms}
replaces the factor $d$ in the linear part by $\sqrt d$ and adds the quadratic
quantity
\[
 \int\norm{x-y}^2\,d\nu(y).
\]
The coupling of \Cref{thm:product-coupling} controls both quantities needed
here. Their entropy bounds are
\[
 \Psi_n(H)
 \quad\text{and}\quad
 \sqrt{nH}+H,
\]
respectively. We now show that the factor $\sqrt{nH}$ in the second bound
cannot be replaced by a smaller order uniformly over the subgaussian class.
For probability measures $\Pp,Q$ on $\R^n$, write
\[
 W_2(\Pp,Q)^2
 =\inf\{\E\norm{X-Y}^2:(X,Y)\text{ is a coupling of }\Pp\text{ and }Q\}.
\]

\begin{proposition}
\label{prop:strong-cost-optimal}
There are product measures $\Pp,Q$ on $\{-1,1\}^n$ such that
$D(Q\Vert \Pp)\asymp H$ and
\begin{equation}\label{eq:strong-cost-lower}
 W_2(\Pp,Q)^2\ge c\sqrt{nH}
\end{equation}
whenever $0<H\le cn$. Thus the term $\sqrt{nH}$ in
\eqref{eq:product-strong} is optimal.
\end{proposition}

\begin{proof}
Let $\Pp_1$ be uniform on $\{-1,1\}$. Put
\[
 Q_1\{1\}=\frac12+\delta,
 \qquad
 Q_1\{-1\}=\frac12-\delta,
\]
where $0<\delta\le1/4$. Then
\[
 D(Q_1\Vert \Pp_1)\asymp\delta^2.
\]
Any coupling must move mass $\delta$ across distance two. Hence
\[
 W_2(\Pp_1,Q_1)^2=4\delta.
\]
Take $\Pp=\Pp_1^{\otimes n}$ and $Q=Q_1^{\otimes n}$. For any coupling
$(X,Y)$ of $\Pp$ and $Q$, each pair $(X_i,Y_i)$ is a coupling of $\Pp_1$ and
$Q_1$, so
\[
 \E\norm{X-Y}^2
 =\sum_{i=1}^n\E|X_i-Y_i|^2
 \ge4n\delta.
\]
The product of the one-dimensional optimal couplings attains equality. Thus
\[
 W_2(\Pp,Q)^2=4n\delta.
\]
Also relative entropy adds over products, and a direct calculation gives
\[
 D(Q\Vert \Pp)=nD(Q_1\Vert \Pp_1)\asymp n\delta^2.
\]
Choose $\delta\asymp\sqrt{H/n}$. This proves
\eqref{eq:strong-cost-lower}.
\end{proof}

\medskip
\noindent\textit{Additional remarks.}

\begin{remark}
If the coordinates are bounded, Talagrand's convex-distance inequality
\cite[Section~4.1, Theorem~4.1.1]{Talagrand} gives
\[
 \norm{Z-\E Z}_{L_p}
 \le C_KLd n^{(d-1)/2}\sqrt p
\]
in the same moderate range. The logarithm in the second term is absent. This agrees with the distinction
between bounded product measures and the full subgaussian class: Huang--Tikhomirov
prove the logarithmic subgaussian upper bound in \cite[Theorem~1.3]{HuangTikhomirov}
and its sharpness in \cite[Proposition~1.4]{HuangTikhomirov}.
\end{remark}

\begin{remark}
If $U$ is uniform on $\sqrt nS^{n-1}$, then
\[
 \norm{x^{\otimes d}-y^{\otimes d}}^2
 =2n^d\left[1-\left(\frac{\ip{x}{y}}n\right)^d\right]
 \le dn^{d-1}\norm{x-y}^2.
\]
Thus the map $x\mapsto x^{\otimes d}$ is $\sqrt d\,n^{(d-1)/2}$-Lipschitz
on the sphere $\sqrt nS^{n-1}$. This fixed-radius calculation is consistent
with the scale $dn^{d-1}$ suggested by the derivative formula.
\end{remark}

\begin{remark}
The same coupling argument applies when independent base vectors are used
with multiplicities $r_1,\ldots,r_m$. For changes parallel to the base
vectors, repeated copies add before squaring and give the parameter
$\sum_j r_j^2$. For changes orthogonal to the base vectors, the squared
contributions add and give $\sum_j r_j$. We do not develop this extension
here.
\end{remark}

\begin{remark}
\Cref{thm:full-range,thm:bounded-full} give convex concentration for one
symmetric tensor in the full natural range. The bounded theorem is the
symmetric analogue of Vershynin's convex concentration theorem for simple
tensors \cite[Theorem~1.3]{Vershynin}. The matching lower bounds in
\Cref{thm:lower,cor:rate-sharpness} show minimax sharpness even under a fixed absolute subgaussian bound.

The coupling bounds themselves are already optimal in the general
subgaussian class by \Cref{prop:strong-cost-optimal}. Stronger estimates for
special subclasses therefore require additional structure beyond these
general coupling bounds. We treat the Euclidean subclass separately.
\end{remark}

\section*{Acknowledgments}
The author is grateful to Professor Hanchao Wang for his helpful discussions
and guidance.

\end{document}